\documentclass[3p]{elsarticle}

\usepackage[latin1]{inputenc}
\usepackage{amssymb}
\usepackage{graphicx}
\usepackage{amsmath}
\usepackage{amsthm}
\usepackage{mathrsfs}
\usepackage{placeins}
\usepackage{tikz}
\usepackage{pgfplots}

\newtheorem{theorem}{Theorem}
\newtheorem{lemma}[theorem]{Lemma}

\theoremstyle{definition}
\newtheorem{definition}[theorem]{Definition}
\newtheorem{remark}[theorem]{Remark}
\newtheorem{example}[theorem]{Example}

\newcommand{\f}[1]{\mathbf{#1}}

\usepackage{color}
\definecolor{dred}{rgb}{0.92,0,0}

\journal{}

\begin{document}

\begin{frontmatter}

\title{Approximation properties over self-similar meshes of curved finite elements and applications to subdivision based isogeometric analysis}

\author[1]{Thomas Takacs}

\affiliation[1]{organization={Johann Radon Institute for Computational and Applied Mathematics (RICAM), Austrian Academy of Sciences},
	addressline={Altenberger Stra{\ss}e 69}, 
	city={Linz},
	postcode={4040}, 
	country={Austria}}
	
\begin{abstract}
In this study we consider domains that are composed of an infinite sequence of self-similar rings and corresponding finite element spaces over those domains. The rings are parameterized using piecewise polynomial or tensor-product B-spline mappings of degree $q$ over quadrilateral meshes. We then consider finite element discretizations which, over each ring, are mapped, piecewise polynomial functions of degree $p$. Such domains that are composed of self-similar rings may be created through a subdivision scheme or from a scaled boundary parameterization.

We study approximation properties over such recursively parameterized domains. The main finding is that, for generic isoparametric discretizations (i.e., where $p=q$), the approximation properties always depend only on the degree of polynomials that can be reproduced exactly in the physical domain and not on the degree $p$ of the mapped elements. Especially, in general, $L^\infty$-errors converge at most with the rate $h^2$, where $h$ is the mesh size, independent of the degree $p=q$. This has implications for subdivision based isogeometric analysis, which we will discuss in this paper.
\end{abstract}


\end{frontmatter}

\section{Introduction}

In this paper we consider a class of self-similar geometry parameterizations that includes scaled boundary parameterizations as in~\cite{arioli2019scaled} and domains constructed via subdivision schemes. While scaled boundary parameterizations are always self-similar by definition, only the characteristic rings of subdivision surfaces are self-similar, cf.~\cite{peters2008subdivision}. In general, the subdivision surface rings around extraordinary vertices become more and more self-similar as the mesh is refined. As examples we study Doo--Sabin subdivision and Catmull--Clark subdivision, introduced in~\cite{doo1978behaviour} and in~\cite{catmull1978recursively}, respectively, in more detail.

Subdivision surfaces are an important tool in constructing smooth shapes from finite control structures. They are used in computer graphics for character animation and simulation, cf.~\cite{derose1998subdivision}, as well as e.g. in engineering applications~\cite{green2002subdivision,cirak2002integrated}. After the advent of isogeometric analysis~\cite{Hughes2005}, which aims at bringing geometric modeling and numerical simulations closer together, interest in subdivision surfaces for simulation was renewed, cf.~\cite{riffnaller2016isogeometric,pan2016isogeometric,zhang2018subdivision}. The recent report~\cite{dietz2023subdivision} has summarized tasks that need to be resolved when using subdivision based discretizations for isogeometric analysis. 

This paper intends to shed some light on Task 4.1 of~\cite{dietz2023subdivision}, the study of error estimates. There, for Catmull--Clark subdivision, the authors reason that away from extraordinary vertices the discretization space generalizes cubic B-splines and the convergence rate for $L^2$-errors (and $L^\infty$) should be $h^4$. The authors moreover speculate that the observed suboptimal convergence rates are due to the reduced polynomial reproduction near extraordinary vertices.

However, it seems that the suboptimal approximation properties of subdivision surfaces are not (only) due to the reduced polynomial reproduction near the extraordinary vertices, but also due to the behavior of the geometry mapping itself. The reason for this lies in the scaling properties of Sobolev norms and seminorms and in the inherent issues related to approximation properties of curved (isoparametric) finite elements, as will be seen later.

This paper is organized as follows. In Section~\ref{sec:domain-mesh} we introduce the types of domains and meshes that we will study. Section~\ref{sec:spaces} gives an overview of the discretization spaces and underlying Sobolev spaces that we consider. The main findings on lower bounds for approximation estimates are given in Section~\ref{sec:approximation}. We then provide several numerical examples on scaled boundary parameterizations and characteristic rings of subdivision surfaces in Section~\ref{sec:num-tests}. The implications for isogeometric methods based on subdivision surfaces are summarized in Section~\ref{sec:subdivision-considerations} and the paper is concluded in Section~\ref{sec:conclusion}.

\section{Domain and mesh structure}
\label{sec:domain-mesh}

We consider an open domain of interest $\Omega\subset\mathbb{R}^2$ that contains the origin $\f o = (0,0)^T \in \Omega$ and that is formed by an infinite sequence of open and ring-shaped subdomains
\[
 \overline{\Omega} = \bigcup_{i=0}^\infty \overline{\Omega^i},
\]
with $\Omega^i \cap \Omega^{i'} = \emptyset$, for all $i\neq i'$. We assume that each ring $\Omega^i$ is composed of $N$ elements $\{\omega^i_n\}_{n=1,\ldots,N}$, with parameterizations
\[
 \f G^i_n : B \rightarrow \omega^i_n,
\]
where $B=\left]0,1\right[^2$. Moreover, we assume that the element parameterizations $\f G^i_n$ are $C^\infty$-smooth, regular and for each $n=1,\ldots,N$ the mappings are scaled versions of the initial element mappings $\f G^0_n$, i.e., $\f G^i_n(B) = \lambda^i \f G^0_n(B)$, for some $\lambda\in\left]0,1\right[$. Consequently, the subdomains $\Omega^i$ are self-similar rings around the origin $\f o$, which satisfy
\[
 \Omega^{i+1} = \lambda \cdot \Omega^{i}.
\]

In Figure~\ref{fig:example-configs} we consider some configurations of practical interest. Note that both the disk and the triangle are scaled boundary parameterizations, where in case of the disk the circular boundary is scaled to the origin at the center and in case of the triangle the top-right edge is the boundary curve and the origin is in the bottom-left corner.
\begin{figure}[!ht]
    \centering
    \includegraphics[height=0.15\textheight]{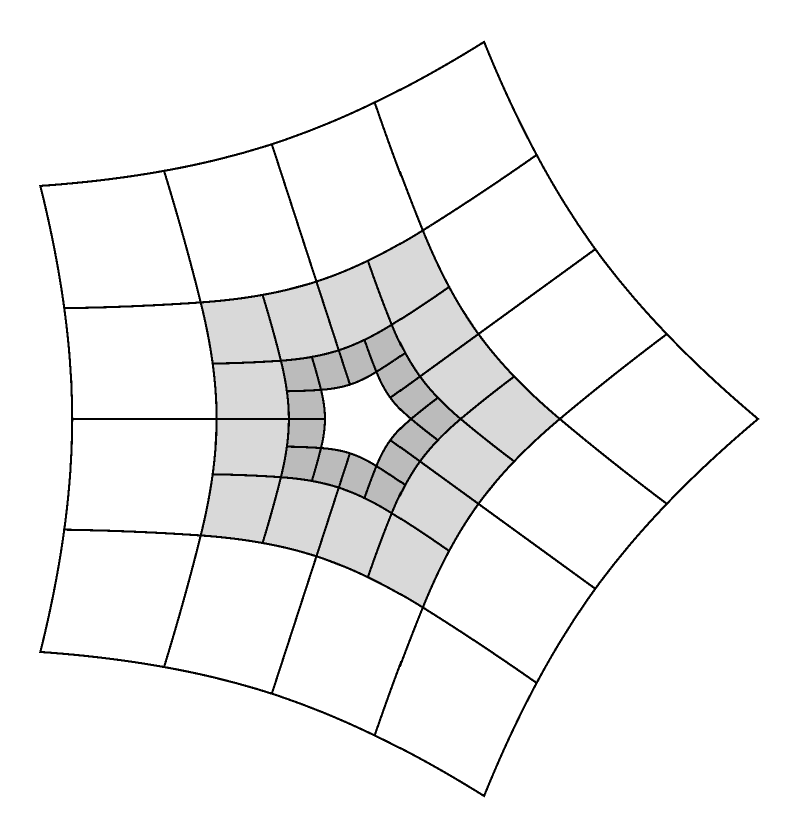} \qquad
    \includegraphics[height=0.15\textheight]{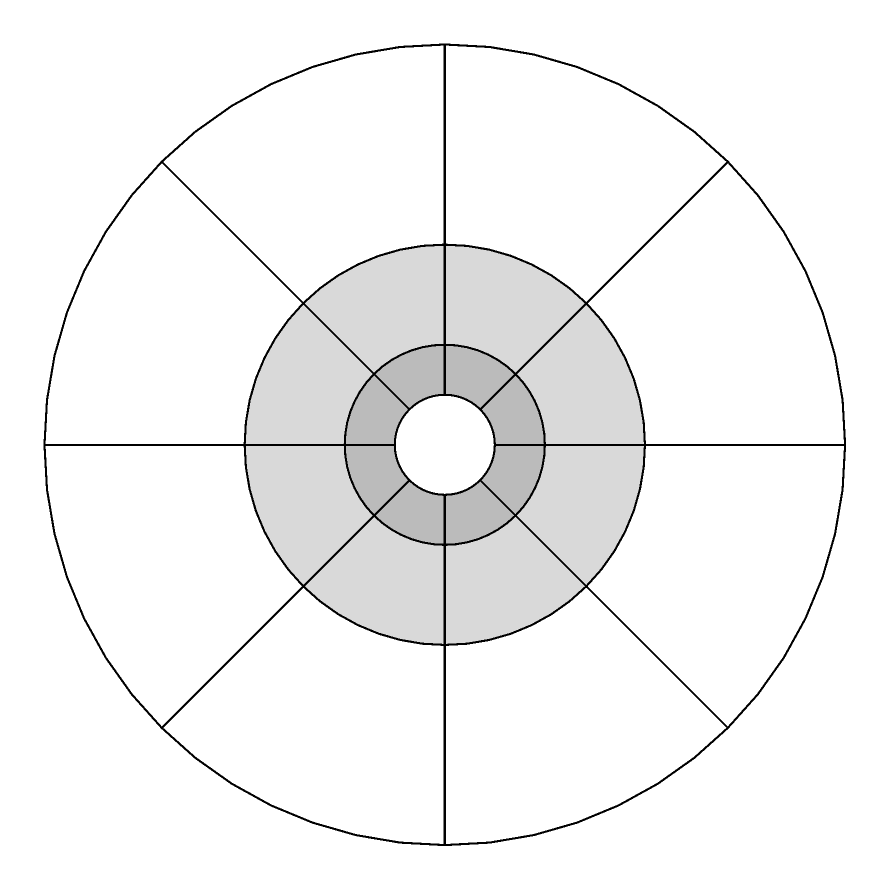} \qquad
    \includegraphics[height=0.15\textheight]{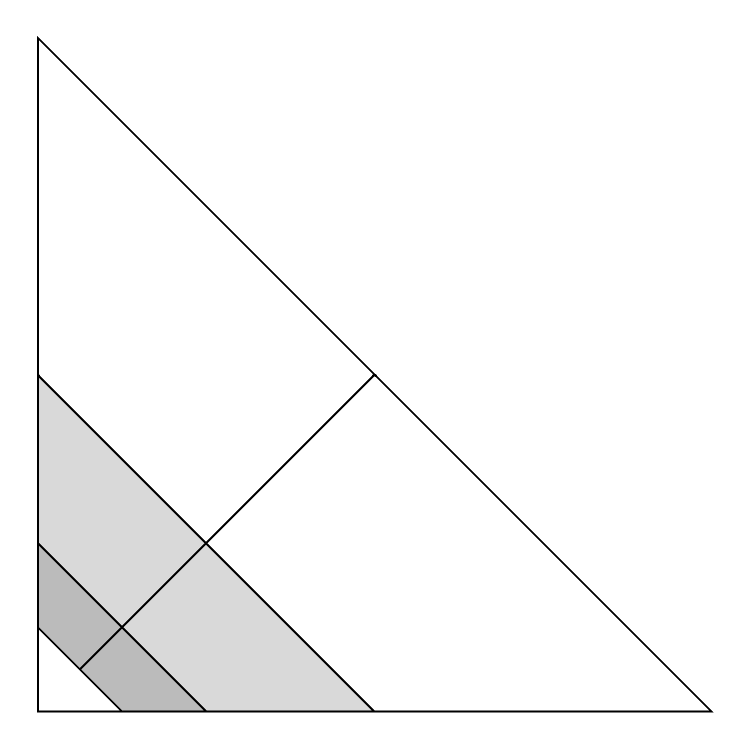}
    \caption{Examples of ring domains ($\Omega^0$ in white, $\Omega^1$ in gray and $\Omega^2$ in dark gray): A characteristic ring of Doo--Sabin subdivision (left), a polar-like parameterization (center) and a triangle, which can be interpreted as a singularly mapped quadrilateral (right). In all three examples we have $\lambda=1/2$.}\label{fig:example-configs}
\end{figure}

Let us now consider, for any number of refinements $k$, a mesh on $B$ by bisection in each direction:
\[
 {\widehat{\mathcal{M}}}^k = \left\{ \left]\frac{j_1}{2^{k}},\frac{j_1+1}{2^{k}}\right[ \times \left]\frac{j_2}{2^{k}},\frac{j_2+1}{2^{k}}\right[ : 0\leq j_1,j_2 \leq 2^{k}-1 \right\}.
\]
We can then define a mesh on $\Omega$ for a fixed level $\ell$ as
\[
 \mathcal{M}^\ell = \bigcup_{i=0}^{\ell} \mathcal{M}^\ell_i,
\]
where on all rings $\Omega^i$, with $0\leq i \leq \ell$, we consider the meshes
\[
 \mathcal{M}^\ell_i = \bigcup_{n=1}^{N} \mathcal{M}^\ell_{i,n},
\]
with
\[
 \mathcal{M}^\ell_{i,n} = \left\{ \f G^i_n\left(b\right) : b \in {\widehat{\mathcal{M}}}^{\ell-i} \right\},
\]
which is obtained by $(\ell-i)$-times bisecting the element $\omega_n^i$ of $\Omega^i$. Moreover, we ignore the structure of the finer rings $\Omega^i$, with $i \geq \ell+1$, and consider the \emph{cap}
\[
 \omega^{\ell+1}_\ast = \mbox{int}\left(\bigcup_{i=\ell+1}^\infty \overline{\Omega^i}\right) = \lambda^{\ell+1} \Omega
\]
as a single element and define the global mesh as
\[
 \mathcal{M}^\ell_\ast = \mathcal{M}^\ell \cup \{\omega^{\ell+1}_\ast\}.
\]
Figure~\ref{fig:example-meshes} depicts meshes $\mathcal{M}^2$ corresponding to the example configurations of Figure~\ref{fig:example-configs}. Note that in case of subdivision surfaces, these are exactly the meshes that one otains through dyadic refinement. However, polar or similar singular parameterizations one would usually refine through standard bisection on the entire patch, cf. Figure~\ref{fig:examples_scaled-bdr_fix} (center right). However, the mesh refinement proposed here yields locally quasi-uniform and shape-regular meshes and, in case of a triangle mapping as in Figure~\ref{fig:example-meshes} (right), optimal approximation can be shown, cf.~\cite{takacs2015approximation}.
\begin{figure}[!ht]
    \centering
    \includegraphics[height=0.15\textheight]{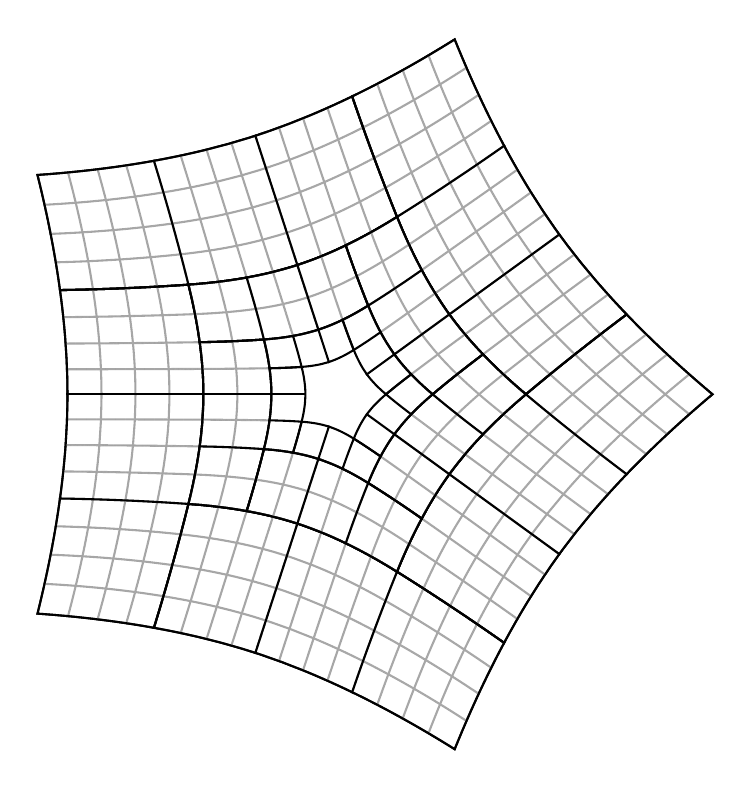} \qquad
    \includegraphics[height=0.15\textheight]{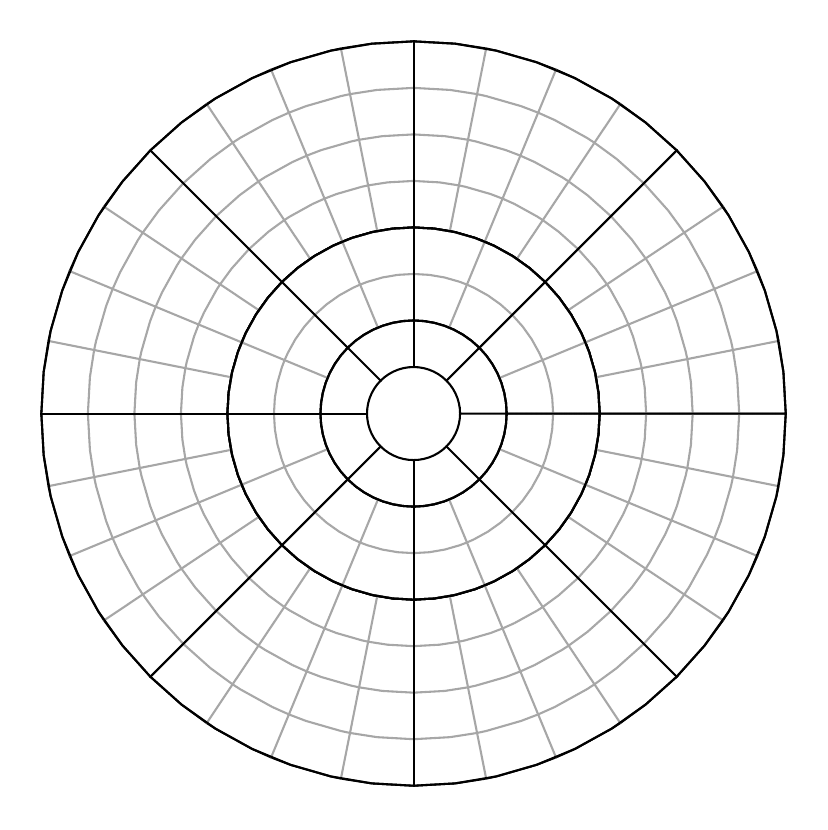} \qquad
    \includegraphics[height=0.15\textheight]{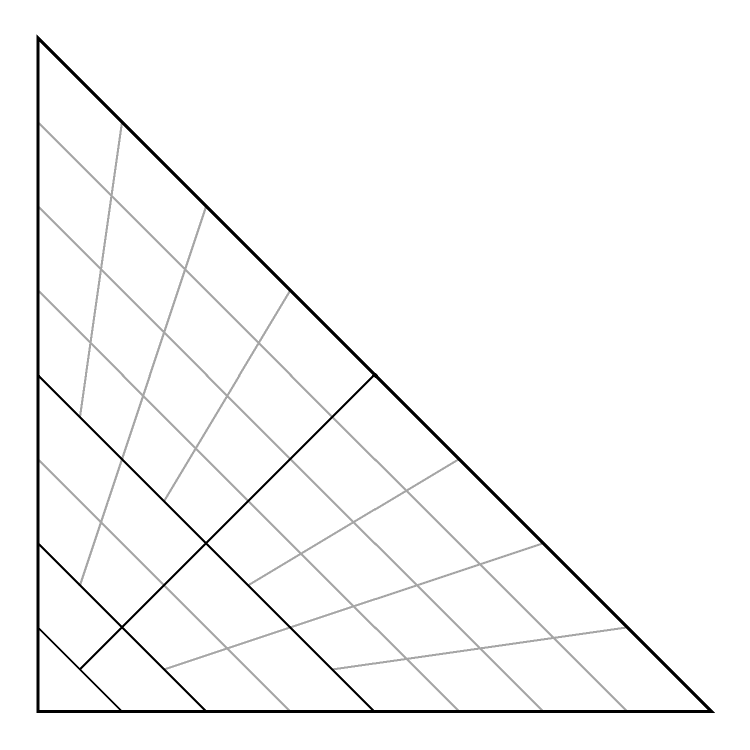}
    \caption{Examples of meshes for $\ell=2$. Note that the further outside the ring, the more refined it is. This results in a quasi-uniform mesh.}\label{fig:example-meshes}
\end{figure}

Even though we consider here only dyadic refinement on each ring $\Omega^i$, all results carry over to other types of refinement as well.
We have the following bounds on the size of mesh elements.
\begin{lemma}\label{lem:element-size}
 For arbitrary $\ell$ and $0\leq i \leq \ell$ we have 
 \[
  \mathrm{diam}(\omega) \sim \frac{1}{2^{\ell-i}}\lambda^i,
 \]
  for all $\omega \in \mathcal{M}^\ell_i$, as well as
 \[
  \mathrm{diam}(\omega^{\ell+1}_\ast) \sim \lambda^{\ell+1}.
 \]
 Here $\mathrm{diam}(\omega)$ denotes the diameter of the element $\omega$. We use the notation $a \lesssim b$ if the exists a constant $c$ which is independent of $\ell$, $i$ and $\lambda$, such that $a \leq c \, b$, and say $a \sim b$ if $a \lesssim b$ and $b \lesssim a$.
\end{lemma}
\begin{proof}
 For each $\omega \in \mathcal{M}^\ell_i$ there exists an $n$, such that
 \[
  \omega = \f G^i_n\left(b\right) = \lambda^i \f G^0_n\left(b\right),
 \]
 where $b \in {\widehat{\mathcal{M}}}^{\ell-i}$. The second equality follows from the definiton of $\f G^i_n$. We have $\mbox{diam}(b) = \sqrt{2}/{2^{\ell-i}}$ and the bounds follow from the regularity of $\f G^0_n$. By definition $\mbox{diam}(\omega^{\ell+1}_\ast) = \lambda^{\ell+1} \mbox{diam}(\Omega)$, which concludes the proof.
\end{proof}
We define the mesh size $h$ of $\mathcal{M}^\ell_\ast$ as $h = \max_{\omega\in \mathcal{M}^\ell_\ast}( \mbox{diam}(\omega) )$. Lemma~\ref{lem:element-size} yields the following bound for the mesh size
\begin{equation}
 h \sim \max ( \lambda,{1}/{2} )^\ell.
\end{equation}
Thus, the meshes are locally quasi-uniform and, for $\lambda=1/2$, the sequence of meshes is quasi-uniform. 

\section{Spaces and scaling properties}
\label{sec:spaces}

In the following we introduce the discretization spaces and Sobolev spaces that we will consider as well as their scaling properties when defined over rings $\Omega^i$.

\subsection{Discretization spaces}

For a given domain $D\subset \mathbb{R}^2$, we denote by $\mathbb{Q}^k[D]$ the space of polynomials of maximum degree $k$ in each direction and by $\mathbb{P}^k[D]$ the space of polynomials of total degree $k$ over $D$. If the domain follows from context, we simply write $\mathbb{Q}^k$ or $\mathbb{P}^k$.

\begin{definition}[Space of piecewise polynomials]
For any $\ell$, $i$ and $n$, with $0 \leq i \leq \ell$ and $1\leq n\leq N$, we define on the element $\omega \in \mathcal{M}^\ell_{i,n}$ the local mapped polynomial space $\mathcal{S}[p;\omega]$ as
\[
 \mathcal{S}[p;\omega] = \{ \varphi : \omega \rightarrow \mathbb{R} \;,\; \varphi \circ \f G^i_n \in \mathbb{Q}^{p}\}.
\]
Moreover, let $\mathcal{S}[p;\omega^{\ell+1}_\ast]$ be a suitable finite-dimensional space on the cap $\omega^{\ell+1}_\ast$. For any subset $\Theta$ of the mesh elements $\Theta \subseteq \mathcal{M}^\ell_\ast$ we define the piecewise polynomial space over $\theta \subseteq \Omega$, with $\overline{\theta}=\bigcup_{\omega \in \Theta} \overline{\omega}$, as 
\[
 \mathcal{S}[p;\Theta] = \{ \varphi \in L^2(\theta) : \varphi|_{\omega} \in \mathcal{S}[p;\omega], \; \mbox{ for all }\omega \in \Theta \}.
\]
\end{definition}
We assume that all element mappings $\f G^i_n$ are polynomials of maximum degree $q$, i.e., $\f G^i_n \in (\mathbb{Q}^{q})^2$. We call the spaces isoparametric, if $p=q$.
\begin{definition}[Reproduction degree]
Let $\omega \in \mathcal{M}^\ell_\ast$. The \emph{reproduction degree} $\kappa = \kappa[\omega] \in \mathbb{Z}$ of the space $\mathcal{S}[p;\omega]$ is the largest integer, such that $\mathbb{P}^{\kappa}[\omega] \subseteq \mathcal{S}[p;\omega]$.
\end{definition}
\begin{lemma}
For all $\omega \in \mathcal{M}^\ell_{i,n}$ and for a fixed degree $p$, the reproduction degree $\kappa[\omega]$ satisfies
\[
 \kappa[\omega] = \kappa[\omega^0_n],
\]
i.e., it is inherited from the coarsest mesh.
\end{lemma}
\begin{proof}
 All elements $\omega \in \mathcal{M}^\ell_{i,n}$ are constructed from $\omega^0_n$ by restricting the mapping $\f G^i_n$ to a subdomain of $B$ and by scaling $\f G^i_n=\lambda^i \f G^0_n$. Neither restricting the parameter domain nor scaling have an effect on the polynomial reproduction properties of the space.
\end{proof}
\begin{example}\label{exa:example-kappa}
In the following we want to give an example for which we explicitly compute the reproduction degree. Let us consider the element mapping
\[
\begin{array}{rccl}
 \f G :& B &\rightarrow& \omega \\
  &(u,v)&\mapsto& (x,y)^T = \left(u+u^2v-u^2v^2,v\right)^T,
\end{array}
\]
which satisfies $\f G\in \mathbb{Q}^{2}$, see Figure~\ref{fig:example-kappa}. 

\begin{figure}[!ht]
    \centering
    \includegraphics[height=0.1\textheight]{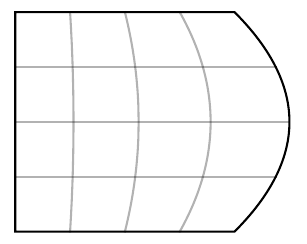}
    \caption{Element $\omega$ from Example~\ref{exa:example-kappa}: The element degree is two and the reproduction degree of the space $\mathcal{S}[p;\omega]$ is $\kappa=\lfloor \frac{p}{2} \rfloor$.}\label{fig:example-kappa}
\end{figure}

Let us study the reproduction degree $\kappa$ of $\mathcal{S}[p;\omega]$ for $p\in\{2,3,4\}$. we consider first a quadratic polynomial $\psi \in \mathbb{P}^{2}[\omega]$, with general coefficients
\[
 \psi(x,y) = c_{0,0} + c_{1,0}\, x + c_{0,1}\, y +c_{2,0}\, x^2 + c_{1,1}\, xy + c_{0,2}\, y^2.
\]
To determine $\kappa$ we need to compute the composition
\[
 \psi \circ \f G(u,v) = c_{0,0} + c_{1,0}\, (u+u^2v-u^2v^2) + c_{0,1}\, v +c_{2,0}\, (u+u^2v-u^2v^2)^2 + c_{1,1}\, (u+u^2v-u^2v^2)v + c_{0,2}\, v^2.
\]
Let us now check whether $\psi \in \mathcal{S}[p;\omega]$ for $p\in\{2,3,4\}$, which is equivalent to checking $\psi \circ \f G \in \mathbb{Q}^p$. We have
\[
 c_{0,0} + c_{1,0}\, (u+u^2v-u^2v^2) + c_{0,1}\, v + c_{0,2}\, v^2 \in \mathbb{Q}^2,
\]
\[
 c_{1,1}\, (u+u^2v-u^2v^2)v \in \mathbb{Q}^3\quad\mbox{ but }\quad c_{1,1}\, (u+u^2v-u^2v^2)v \notin \mathbb{Q}^2,
\]
and
\[
 c_{2,0}\, (u+u^2v-u^2v^2)^2 \in \mathbb{Q}^4\quad\mbox{ but }\quad c_{2,0}\, (u+u^2v-u^2v^2)^2 \notin \mathbb{Q}^3,
\]
for general non-zero coefficients $c_{i,j}$. 
Thus $\mathcal{S}[2;\omega]$ only reproduces linears and one quadratic term, consequently the reproduction degree $\kappa[\omega]=1$ in this case. The space $\mathcal{S}[3;\omega]$ reproduces linears and two quadratic terms, nevertheless one quadratic term cannot be reproduced and $\kappa[\omega]=1$ also for $\mathcal{S}[3;\omega]$. However, for $p=4$ all quadratic terms can be reproduced and $\kappa[\omega]=2$ for $\mathcal{S}[4;\omega]$. It is easy to see that in this example we have $\kappa[\omega]=\lfloor \frac{p}{2} \rfloor$ for all $\mathcal{S}[p;\omega]$.
\end{example}
All polynomials of total degree $\kappa[\omega]$ in physical coordinates can be reproduced by functions from the local space $\mathcal{S}[p;\omega]$. For all elements $\omega$ we have $\kappa[\omega] \geq \lfloor \frac{p}{q} \rfloor$. If $\omega$ is a generic element of bi-degree $(q,q)$, then $\kappa[\omega] = \lfloor \frac{p}{q} \rfloor$. In our further analysis we will ignore the precise definition of the space $\mathcal{S}[p;\omega^{\ell+1}_\ast]$ on the cap, as it does not really matter for the global convergence properties. What matters is only the reproduction degree $\kappa[\omega^{\ell+1}_\ast]$.

\subsection{Sobolev spaces and scaling relations}

For a given, open domain $\theta\subset \mathbb{R}^2$, with sufficiently regular boundary, we denote by $H^{r}(\theta)$ the Sobolev space of order $r$, that is, the subspace of $L^2(\theta)$ where all derivatives up to order $r$ are in $L^2(\theta)$ as well. Let $\Theta$ be a mesh on $\theta$, then we denote by $\mathcal{H}^r(\theta;\Theta)$ the broken Sobolev space, which has the norm
\[
\|\varphi \|_{\mathcal{H}^r(\theta;\Theta)}^2= \sum_{\sigma=0}^r |\varphi |_{\mathcal{H}^\sigma(\theta;\Theta)}^2
\]
and seminorms
\[
|\varphi |_{\mathcal{H}^\sigma(\theta;\Theta)}^2= \sum_{\omega \in \Theta} |\varphi |_{H^\sigma(\omega)}^2.
\]
Here we denote with $|\varphi|_{H^\sigma(\omega)}$ the standard Sobolev seminorm on the element $\omega$.

Sobolev seminorms scale with the size of the domain. 
\begin{lemma}\label{lem:scaling}
Let $\omega = \mu \omega^0$ and let $r\geq 0$. Then we have
\begin{equation}\label{eq:scaling-seminorms}
 |\varphi|_{H^r(\omega)} = \mu^{1-r} |\varphi \circ \f m|_{H^r(\omega^0)},
\end{equation}
where $\f m : \omega^0 \rightarrow \omega$, with $\f m(\f x) = \mu \f x$. 
\end{lemma}
\begin{proof}
Let $\f m(X,Y) = (\mu X, \mu Y) = (x,y)$. Then this substitution of variables yields
\[
 |\varphi|_{H^r(\omega)}^2 = \int_{\omega} \sum_{\alpha_x+\alpha_y=r} \left( \frac{\partial}{\partial x^{\alpha_x}}\frac{\partial}{\partial y^{\alpha_y}} \varphi \right)^2 \; \mathrm{d}\f x = \int_{\omega^0} \sum_{\alpha_x+\alpha_y=r} \left( \frac{1}{\mu^r} \frac{\partial}{\partial X^{\alpha_x}}\frac{\partial}{\partial Y^{\alpha_y}} \Phi \right)^2 \mu^2 \; \mathrm{d}\f X = \mu^{2-2r} |\Phi|_{H^r(\omega^0)}^2,
\]
where $\Phi = \varphi \circ \f m$. 
\end{proof}
We will use these scaling relations to study approximation errors of the piecewise polynomial spaces defined over self-similar meshes as defined above.

\section{Approximation properties}
\label{sec:approximation}

In this section we study approximation properties of the spaces $\mathcal{S}[p;\mathcal{M}^\ell_\ast]$. We first discuss the general setting and then derive explicit lower bounds for convergence rates in several norms.

\subsection{The general setting}

When studying approximation properties of piecewise polynomial spaces and spline spaces for numerical analysis, one is usually interested in estimates of the form:
\begin{equation}\label{eq:approximation-template}
 \inf_{\varphi_h \in \mathcal{S}^\ell} | \varphi - \varphi_h |_{H^{r}(\Omega)} \lesssim \varrho(\ell) \, \|\varphi\|_{H^s(\Omega)},
\end{equation}
for all $\varphi \in H^s(\Omega)$, with $0\leq r < s\leq p+1$. Here 
\begin{itemize}
 \item the space $\mathcal{S}^\ell$ equals $\mathcal{S}[p;\mathcal{M}^\ell_\ast]$ or is a subspace of the piecewise polynomial space, usually assuming some order of continuity, e.g., $\mathcal{S}^\ell = \mathcal{S}[p;\mathcal{M}^\ell_\ast]\cap C^k(\Omega\setminus \omega^{\ell+1}_\ast)$, for some $k$,
 \item $\varrho(\ell)$ is the convergence rate as a function of the level $\ell$ (alternatively described with respect to the mesh size $h$) and
 \item the bound might depend in some non-trivial way on the shape and parameterization of $\Omega$.
\end{itemize}
In this study, we are interested in a maximum rate $\varrho$, such that such an estimate can exist, i.e., we find a function $\varphi^* \in H^{p+1}(\Omega)$ (actually a global polynomial in physical coordinates), such that an estimate
\begin{equation}\label{eq:approximation-worst-case}
 \inf_{\varphi_h \in \mathcal{S}[p;\mathcal{M}^\ell_\ast]} | \varphi^* - \varphi_h |_{\mathcal{H}^{r}(\Omega;\mathcal{M}^\ell_\ast)} \gtrsim \varrho(\ell) \, \|\varphi^*\|_{H^{p+1}(\Omega)}
\end{equation}
exists. 

In the following we focus on the approximation properties of piecewise polynomials on the self-similar rings $\Omega^i$ and not so much on the possibly suboptimal approximation properties in the cap $\omega^{\ell+1}_\ast$. The reason for this is that at the cap it is relatively easy and cheap to enrich the space sufficiently, to raise the reproduction degree $\kappa[\omega^{\ell+1}_\ast]$. So, for all practical purposes, one can assume that the error at the cap is negligible. However, if a specific construction is studied, such as subdivision surfaces as discussed in Section~\ref{sec:subdivision-considerations} or almost-$C^1$ splines as introduced in~\cite{takacs2023almost}, then the actual reproduction degree at the cap is of importance.

\subsection{Approximation properties on a single ring}

When studying the approximation properties on a single ring, we can follow the theory on isoparametric finite elements as developed in~\cite{ciarlet1972interpolation,ciarlet2002finite} or on isogeometric discretizations as in~\cite{Bazilevs2006,beirao2014actanumerica}, e.g., cf.~\cite[Theorem 3.2]{Bazilevs2006}. Thus, considering tensor-product B-splines on the single ring $\Omega^0$, we have the following.
\begin{theorem}\label{thm:quasi-interpolant}
 Let $0\leq r < s\leq p+1$ and let $\varphi \in H^{s}(\Omega^0)$. We assume that there exists a spline space $\mathcal{S}^\ell \subseteq \mathcal{S}[p;\mathcal{M}^\ell_0] \cap H^{r}(\Omega^0)$, which locally reproduces mapped polynomials, and a projector $\Pi^{\ell}: L^{2}(\Omega^0) \rightarrow \mathcal{S}^\ell$, which is locally bounded in $L^2(\Omega^0)$. Then we have
\begin{equation}
 \inf_{\varphi_h \in \mathcal{S}^\ell} | \varphi - \varphi_h |_{H^r(\Omega^0)} \leq | \varphi - \Pi^{\ell}(\varphi) |_{H^r(\Omega^0)} \lesssim \left(\frac{1}{2^\ell}\right)^{s-r} \, \|\varphi\|_{H^{s}(\Omega^0)}.
\end{equation}
\end{theorem}
Such an estimate follows directly from~\cite[Theorem 3.2]{Bazilevs2006}. 

We are not going into the details of splines, isogeometric discretizations and spline projectors here. The important thing to note is that on many mesh configurations over ring-shaped domains such discretizations can be constructed (e.g. for the examples depicted in Figure~\ref{fig:example-meshes}). The existence of such splines on the ring $\Omega^0$ (and consequently also on all finer rings) requires the element parameterizations to match with sufficient continuity. E.g. the splines defined over the rings obtained through Catmull--Clark subdivision are $C^2$, thus the assumptions of Theorem~\ref{thm:quasi-interpolant} are satisfied for all $r\leq 3$. 

It is enticing to assume that these estimates on single rings induce similar estimates on the entire domain. However, going into the details of the proofs in~\cite{Bazilevs2006}, which are based on the proofs for isoparametric finite elements in~\cite{ciarlet1972interpolation}, one can find a non-trivial geometry dependence which is here hidden in the unwritten constants of the inequality. 
\begin{remark}
Following the reasoning of~\cite{ciarlet1972interpolation}, approximation errors for isoparametric finite elements converge optimally, that is, with rate $h^{s-r}$, if the local elements converge to bilinear elements sufficiently fast as their size goes to zero. This is true for the mesh $\mathcal{M}^\ell_0$ on a single ring $\Omega^0$, but not for the global mesh $\mathcal{M}^\ell$.
\end{remark}
In the following subsection we study the best approximation to polynomials in physical coordinates.

\subsection{Lower bounds for approximation errors in broken Sobolev seminorms}

Let $p$ be arbitrary but fixed and let us consider estimates for functions in the Sobolev space $H^{p+1}(\Omega)$. In the following we study bounds of the approximation error measured in a broken $\mathcal{H}^r$-seminorm, i.e., we consider the error 
\[
 \inf_{\varphi_h \in \mathcal{S}[p;\mathcal{M}^\ell_\ast]}|\varphi - \varphi_h |_{\mathcal{H}^r(\Omega;\mathcal{M}^\ell_\ast)}.
\]

Consider first the initial mesh $\mathcal{M}^0$ (without the cap) and let $\underline{\kappa}_0$ be the smallest reproduction degree of all elements in the initial mesh, i.e., 
\[
 \underline{\kappa}_0 = \min_{\omega\in \mathcal{M}^0} (\kappa[\omega]).
\]
Let moreover $\kappa[\omega^1_\ast] \geq \underline{\kappa}_0$ (this can be achieved by enriching the space sufficiently). We define the space
\begin{equation}
 \mathbb{P}^{\underline{\kappa}_0+1}_{\max}[\Omega] = \mbox{span}\{ x^{\underline{\kappa}_0+1}, x^{\underline{\kappa}_0}y, x^{\underline{\kappa}_0-1}y^2, \ldots, y^{\underline{\kappa}_0+1} \}
\end{equation}
of polynomials in physical coordinates.

\begin{lemma}\label{lem:maximizing-polynomial-Sobolev}
 Let $0 \leq r \leq \underline{\kappa}_0+1$. There exists a constant $\underline{C}_0^r$, which satisfies
\[
 \sup_{\varphi \in \mathbb{P}^{\underline{\kappa}_0+1}_{\max}[\Omega]} \inf_{\varphi_h \in \mathcal{S}[p;\mathcal{M}^0_\ast]} \frac{|\varphi - \varphi_h |_{\mathcal{H}^r(\Omega;\mathcal{M}^0_\ast)}}{\| \varphi \|_{{H}^{p+1}(\Omega)}} = \underline{C}_0^r > 0.
\]
Moreover, the supremum is attained at $\varphi^{r,\ast} \in \mathbb{P}^{\underline{\kappa}_0+1}_{\max}[\Omega]$, with $\| \varphi^{r,\ast} \|_{{H}^{p+1}(\Omega)} = 1$, satisfying 
\[
 \inf_{\varphi_h \in \mathcal{S}[p;\mathcal{M}^0_\ast]} |\varphi^{r,\ast} - \varphi_h |_{\mathcal{H}^r(\Omega;\mathcal{M}^0_\ast)} = \underline{C}_0^r.
\]
\end{lemma}
\begin{proof}
 Since the error function is a broken Sobolev seminorm, it can be evaluated locally on all elements. For all $0\leq r \leq \underline{\kappa}_0+1$ a function $\varphi \in \mathbb{P}^{\underline{\kappa}_0+1}_{\max}[\Omega]$ exists, which cannot be reproduced on at least one element of $\underline{\mathcal{M}}^0_\ast$. If a function cannot be reproduced, its element-local ${H}^r$-seminorm cannot vanish. This can be shown by contradiction: Assume that $|\varphi- \varphi_h|_{H^r(\omega)}=0$ for some $\varphi_h \in \mathcal{S}[p;\omega]$ and $\varphi \notin \mathcal{S}[p;\omega]$. Then we have
 \[
  \varphi- \varphi_h \in \mathbb{P}^{r-1} \subseteq \mathbb{P}^{\underline{\kappa}_0} \subseteq \mathcal{S}[p;\omega],
 \]
 which implies $\varphi \in \mathcal{S}[p;\omega]$, contradicting our assumption.
 Thus, the seminorm cannot vanish and the supremum $\underline{C}_0^r$ is positive and, since the space is closed, there exists a $\varphi^{r,\ast}$ where it is attained.
\end{proof}
In the following we combine two different estimates. On the one hand, an estimate for the approximation errors on the outer ring $\Omega^0$, as $\ell$ goes to infinity, on the other hand a scaling relation between the error on the ring $\Omega^i$ and on the ring $\Omega^0$, which is based on a scaling of the domain.
\begin{lemma}
There exists a constant $\underline{c}^r$, such that
\[
 \inf_{\varphi_h \in \mathcal{S}[p;\mathcal{M}^\ell_0]} |\varphi^{r,\ast} - \varphi_h |_{\mathcal{H}^r(\Omega;\mathcal{M}^\ell_0)} \geq \underline{c}^r \left(\frac{1}{2^\ell}\right)^{p+1-r} \underline{C}_0^r,
\]
for all $\ell \geq 0$.
\end{lemma}
\begin{proof}
The estimate is a standard bound derived from the optimal convergence rate $\left(\frac{1}{2^\ell}\right)^{p+1-r}$ on the ring $\Omega^0$. Since the polynomial $\varphi^{r,\ast}$ cannot be reproduced exactly, the error converges as stated. 
\end{proof}
\begin{lemma}
We have
\[
 \inf_{\varphi_h \in \mathcal{S}[p;\mathcal{M}^\ell_i]} |\varphi^{r,\ast} - \varphi_h |_{\mathcal{H}^r(\Omega^i;\mathcal{M}^\ell_i)} = \lambda^{i(\underline{\kappa}_0+2-r)} \inf_{\varphi_h \in \mathcal{S}[p;\mathcal{M}^{\ell-i}_0]} |\varphi^{r,\ast} - \varphi_h |_{\mathcal{H}^r(\Omega^0;\mathcal{M}^{\ell-i}_0)}
\]
for all $0\leq i \leq \ell$, and
\[
 \inf_{\varphi_h \in \mathcal{S}[p;\omega^{\ell+1}_\ast]} |\varphi^{r,\ast} - \varphi_h |_{{H}^r(\omega^{\ell+1}_\ast)} = \lambda^{\ell(\underline{\kappa}_0+2-r)} \inf_{\varphi_h \in \mathcal{S}[p;\omega^{1}_\ast]} |\varphi^{r,\ast} - \varphi_h |_{{H}^r(\omega^{1}_\ast)}.
\]
\end{lemma}
\begin{proof}
Since the error functions are piecewise defined, to show the first equation, we can consider a single element $\omega^i \in \mathcal{M}^\ell_i$. From Lemma~\ref{lem:scaling} with $\mu = \lambda^i$ we obtain
\[
 \inf_{\varphi_h \in \mathcal{S}[p;\omega^i]}|\varphi^{r,\ast} - \varphi_h |_{{H}^r(\omega^i)} =
 \mu^{1-r} \inf_{\varphi_h \in \mathcal{S}[p;\omega^i]} |\varphi^{r,\ast} \circ \f m - \varphi_h \circ \f m|_{H^r(\omega^0)},
\]
where $\mu \omega^0 = \omega^1$. In addition, for a monomial $\varphi = x^{\alpha} y^{\beta}$ we have the relation
\begin{equation}\label{eq:scaling-monomials}
 \varphi \circ \f m = \mu^{\alpha+\beta}\varphi.
\end{equation}
Thus, we get $\varphi^{r,\ast} \circ \f m = \mu^{\underline{\kappa}_0+1} \varphi^{r,\ast}$, since $\varphi^{r,\ast} \in \mathbb{P}^{\underline{\kappa}_0+1}_{\max}[\Omega]$. So we get
\begin{eqnarray}
 \inf_{\varphi_h \in \mathcal{S}[p;\omega^i]}|\varphi^{r,\ast} - \varphi_h |_{{H}^r(\omega^i)} &=&
 \mu^{\underline{\kappa}_0+2-r} \inf_{\varphi_h \in \mathcal{S}[p;\omega^i]} |\varphi^{r,\ast} - \mu^{-\underline{\kappa}_0-1}\varphi_h \circ \f m|_{H^r(\omega^0)} \nonumber \\
 &=& \mu^{\underline{\kappa}_0+2-r} \inf_{\hat{\varphi}_h \in \mathcal{S}[p;\omega^0]} |\varphi^{r,\ast} - \hat{\varphi}_h |_{H^r(\omega^0)},\nonumber
\end{eqnarray}
since the polynomial spaces are the same, i.e., for each $\varphi_h \in \mathcal{S}[p;\omega^1]$ there exists a $\hat{\varphi}_h \in \mathcal{S}[p;\omega^0]$ with $\varphi_h \circ \f m = \mu^{\underline{\kappa}_0+1}\hat{\varphi}_h$ and vice versa. The same scaling can be done for the error on the cap $\omega^{\ell+1}_\ast$, with scaling factor $\mu=\lambda^\ell$, which completes the proof.
\end{proof}
We can now prove a global lower bound.
\begin{theorem}\label{thm:global-lower}
Let $0\leq r \leq \underline{\kappa}_0+1$ and let $\varphi^{r,\ast}$ be as in Lemma~\ref{lem:maximizing-polynomial-Sobolev}. Then we have
\[
 \inf_{\varphi_h \in \mathcal{S}[p;\mathcal{M}^\ell_\ast]}|\varphi - \varphi_h |_{\mathcal{H}^r(\Omega;\mathcal{M}^\ell_\ast)} \geq \min(1,\underline{c}^r) \underline{C}_0^r \left( \sum_{i=0}^{\ell} A^{2i} B^{2(\ell-i)} \right)^{1/2},
\]
where $A=\lambda^{\underline{\kappa}_0+2-r}$ and $B={1}/{2^{p+1-r}}$.
\end{theorem}
\begin{proof}
The error function can be split into local contributions as
\[
 \inf_{\varphi_h \in \mathcal{S}[p;\mathcal{M}^\ell_\ast]}|\varphi - \varphi_h |^2_{\mathcal{H}^r(\Omega;\mathcal{M}^\ell_\ast)} = \sum_{i=0}^{\ell+1} (e^\ell_i)^2,
\]
where
\[
 e^\ell_i = \inf_{\varphi_h \in \mathcal{S}[p;\mathcal{M}^\ell_i]}|\varphi - \varphi_h |_{\mathcal{H}^r(\Omega^i;\mathcal{M}^\ell_i)},
\]
for $i=0,\ldots,\ell$, and
\[
 e^\ell_{\ell+1} = \inf_{\varphi_h \in \mathcal{S}[p;\omega^{\ell+1}_\ast]}|\varphi - \varphi_h |_{{H}^r(\omega^{\ell+1}_\ast)}.
\]
Thus, we can bound 
\[
 e^\ell_i = \lambda^{i(\underline{\kappa}_0+2-r)} e^{\ell-i}_0 \geq \lambda^{i(\underline{\kappa}_0+2-r)} \underline{c}^r \left(\frac{1}{2^{\ell-i}}\right)^{p+1-r} \underline{C}_0^r = A^i B^{\ell-i} \; \underline{c}^r\underline{C}_0^r,
\]
for all $0 \leq i \leq \ell-1$, as well as
\[
 e^\ell_\ell = \lambda^{\ell(\underline{\kappa}_0+2-r)} e^{0}_0 \qquad \mbox{and}\qquad e^\ell_{\ell+1} = \lambda^{\ell(\underline{\kappa}_0+2-r)} e^{0}_1,
\]
which satisfy
\[
 (e^\ell_\ell)^2 + (e^\ell_{\ell+1})^2 = A^{2\ell} \left((e^{0}_0)^2+ (e^{0}_1)^2 \right) = A^{2\ell}(\underline{C}_0^r)^2.
\]
Hence, we obtain
\[
 \sum_{i=0}^{\ell+1} (e^\ell_i)^2 \geq \sum_{i=0}^{\ell-1} A^{2i} B^{2(\ell-i)} \; (\underline{c}^r\underline{C}_0^r)^2 + A^{2\ell}(\underline{C}_0^r)^2 \geq \min(1,(\underline{c}^r)^2) (\underline{C}_0^r)^2 \sum_{i=0}^{\ell} A^{2i} B^{2(\ell-i)}
\]
and the result follows.
\end{proof}
We assume here that the function $\varphi^{r,\ast}$ is normalized with respect to the $H^{p+1}$-norm. But note that on the finite-dimensional space $\mathbb{P}^{k_0+1}_{\max}[\Omega]$ all norms are equivalent, so any functional that is a norm on the polynomial space $\mathbb{P}^{k_0+1}_{\max}[\Omega]$ may be on the right hand side of the estimate.
\begin{remark}\label{rem:rates-summary}
Considering the setting of Theorem~\ref{thm:global-lower} we can distinguish three cases:
\begin{itemize}
 \item[(a)] $\lambda^{\underline{\kappa}_0+2-r}>1/2^{p+1-r}$
 \item[(b)] $\lambda^{\underline{\kappa}_0+2-r}=1/2^{p+1-r}$
 \item[(c)] $\lambda^{\underline{\kappa}_0+2-r}<1/2^{p+1-r}$
\end{itemize}
In case (a) the dominating term in Theorem~\ref{thm:global-lower} is $A$ and the estimate simplifies to 
\[
 \inf_{\varphi_h \in \mathcal{S}[p;\mathcal{M}^\ell_\ast]}|\varphi - \varphi_h |_{\mathcal{H}^r(\Omega;\mathcal{M}^\ell_\ast)} \gtrsim \left( \sum_{i=0}^{\ell} A^{2\ell} ({B}/{A})^{2(\ell-i)} \right)^{1/2} \sim A^\ell,
\]
where the equivalence on the right follows from 
\[
 A^{\ell} \leq \left( \sum_{i=0}^{\ell} A^{2\ell} ({B}/{A})^{2(\ell-i)} \right)^{1/2} \leq A^{\ell} \left(\frac{1}{1-(B/A)^2}\right)^{1/2}.
\]
Hence, the lower bound for the convergence rate is $\lambda^{\ell(\underline{\kappa}_0+2-r)}$. 

Equivalently, in case (c) the dominating term is $B$ and the lower bound for the convergence rate is $1/2^{\ell(p+1-r)}$, which is the optimal expected rate. 

In case (b) we have $A=B$, thus the bound in Theorem~\ref{thm:global-lower} yields 
\[
 \inf_{\varphi_h \in \mathcal{S}[p;\mathcal{M}^\ell_\ast]}|\varphi - \varphi_h |_{\mathcal{H}^r(\Omega;\mathcal{M}^\ell_\ast)} \gtrsim \left( \sum_{i=0}^{\ell} A^{2i} A^{2(\ell-i)} \right)^{1/2} = \left( \sum_{i=0}^{\ell} A^{2\ell} \right)^{1/2} = A^\ell \left( \sum_{i=0}^{\ell} 1 \right)^{1/2},
\] 
and the resulting rate is $\sqrt{\ell+1} (1/2)^{\ell(p+1-r)}$, which is slightly suboptimal.
\end{remark}

Theorem~\ref{thm:global-lower} shows that the choice of the polynomial space $\mathbb{P}^{\underline{\kappa}_0+1}_{\max}$ is crucial. While all lower order terms, i.e., all functions from $\mathbb{P}^{\underline{\kappa}_0}$, can be reproduced exactly, resulting in an error $\underline{C}_0^r = 0$, the errors produced by all higher order terms $\mathbb{P}^{{\kappa}}_{\max}$, with $\kappa>\underline{\kappa}_0+1$, will go to zero faster. Thus, when considering the Talyor expansion around the origin of any given function, the terms in the expansion that will dominate the error are the contributions from $\mathbb{P}^{\underline{\kappa}_0+1}_{\max}$.

\subsection{Approximation properties in $L^\infty$}

Using similar scaling arguments as for the $L^2$ bounds, we can show $L^\infty$-error bounds.
\begin{lemma}\label{lem:maximizing-polynomial}
Let $\varphi^{r,\ast}$ be as in Lemma~\ref{lem:maximizing-polynomial-Sobolev}. Then we have
\[
 \inf_{\varphi_h \in \mathcal{S}[p;\mathcal{M}^0]}  \| \varphi^{r,\ast} - \varphi_h \|_{L^\infty(\Omega^0)} = \underline{C}_0 > 0.
\]
\end{lemma}
\begin{proof}
Since $\varphi^{r,\ast}$ cannot be reproduced by $\mathcal{S}[p;\mathcal{M}^0]$ everywhere, the best approximation does not vanish. Therefore the $L^\infty$-norm of the error $\underline{C}_0$ must be non-zero.
\end{proof}
Instead of analyzing the $L^\infty$-error of the function $\varphi^{r,\ast}$, we could also take the function in $\mathbb{P}^{\underline{\kappa}_0+1}_{\max}[\Omega]$ maximizing the error for the best approximation in $L^\infty$. The results will be the same, only with a different constant $\underline{C}_0$. We have the following.
\begin{theorem}\label{thm:Linfty-estimate}
Let $\varphi^{r,\ast}$ be as in Lemma~\ref{lem:maximizing-polynomial-Sobolev}. There exists a constant $\underline{c}_0 > 0$, such that
 \[
  \inf_{\varphi_h \in \mathcal{S}[p;\mathcal{M}^\ell]} \| \varphi^{r,\ast} - \varphi_h \|_{L^\infty(\Omega)} \geq \max(\underline{C}_0 \lambda^{\underline{\kappa}_0+1},\underline{c}_0 1/2^{p+1} )^\ell.
 \]
\end{theorem}
\begin{proof}
 Let $\omega^0_n \in \mathcal{M}^0$ be such that $\inf_{\varphi_h \in \mathcal{S}[p;\omega^0_n]}  \| \varphi^{r,\ast} - \varphi_h \|_{L^\infty(\omega^0_n)} = \underline{C}_0$ and let $\f G^0_n$ be its parameterization. By definition, the scaled element $\omega^\ell_n = \lambda^\ell \omega^0_n$ is an element of the mesh $\mathcal{M}^\ell$. Using the mapping $\f m : \omega^0_n \rightarrow \omega^\ell_n$ between the initial and fine element with scaling factor $\mu=\lambda^\ell$ we obtain
 \[
  \inf_{\varphi_h \in \mathcal{S}[p;\omega^\ell_n]} \| \varphi^{r,\ast} - \varphi_h \|_{L^\infty(\omega^\ell_n)} = \inf_{\varphi_h \in \mathcal{S}[p;\omega^\ell_n]} \| \varphi^{r,\ast} \circ \f m - \varphi_h  \circ \f m \|_{L^\infty(\omega^0_n)}.
 \]
 Using again \eqref{eq:scaling-monomials} we get
 \[
  \inf_{\varphi_h \in \mathcal{S}[p;\omega^\ell_n]} \| \varphi^{r,\ast} - \varphi_h \|_{L^\infty(\omega^\ell_n)} = \mu^{\underline{\kappa}_0+1} \inf_{\varphi_h \in \mathcal{S}[p;\omega^\ell_n]} \| \varphi^{r,\ast} - \mu^{-\underline{\kappa}_0-1}\varphi_h  \circ \f m \|_{L^\infty(\omega^0_n)} = \lambda^{\ell(\underline{\kappa}_0+1)}\underline{C}_0.
 \]
 The bound 
 \[
  \inf_{\varphi_h \in \mathcal{S}[p;\mathcal{M}^\ell]} \| \varphi^{r,\ast} - \varphi_h \|_{L^\infty(\Omega)} \geq {c}_0 1/2^{\ell(p+1)},
 \]
 for some ${c}_0>0$, is a standard estimate for polynomial approximation in $L^\infty$. Thus the result follows.
\end{proof}
\begin{remark}
Assuming again generic isoparametric elements, i.e., $\underline{\kappa}_0=1$, the best possible rate of the $L^\infty$-error is bounded by $\lambda^{2\ell}$, which is suboptimal for $\lambda > 2^{-(p+1)/2}$. Hence, the rate can only be optimal if $\lambda \leq 2^{-(p+1)/2}$, i.e., if the rings shrink fast enough.
\end{remark}

\subsection{Extension to higher dimensional domains}\label{subsec:higher-dim}

There is no reason to restrict this study to planar domains, which we have done only to keep the presentation simple and more easily readable. All results extend also to domains $\Omega\subset \mathbb{R}^d$ of any dimension~$d$, which can be formed by rings $\Omega^i$ that are composed of mapped boxes $\{\omega^i_n\}_{n=1}^N$ via
\[
 \f G^i_n : B \rightarrow \omega^i_n,
\]
where $B=\left]0,1\right[^d$. Instead of the scaling relation~\eqref{eq:scaling-seminorms} in Lemma~\ref{lem:scaling} we have
\begin{equation*}
 |\varphi|_{H^r(\omega)} = \mu^{d/2-r} |\varphi \circ \f m|_{H^r(\omega^0)}.
\end{equation*}
Theorem~\ref{thm:global-lower} then holds with $A=\lambda^{\underline{\kappa}_0+d/2+1-r}$. However, the $L^\infty$-estimate in Theorem~\ref{thm:Linfty-estimate} remains unchanged.

\section{Numerical tests}
\label{sec:num-tests}

In this section we verify some of the theoretical findings with numerical experiments.

\subsection{Scaled boundary parameterizations}

In this subsection we compare two examples of scaled boundary parameterizations, where convergence rates are suboptimal, and also propose remedies. Sobolev regularity properties of isogeometric discretizations over such domains were studied in~\cite{Takacs2011,Takacs2012,Takacs2014}. However, approximation estimates could so far only be shown for parameterizations derived from singularly parameterized triangles, as in Figure~\ref{fig:example-configs} (right), see~\cite{takacs2015approximation}.

\paragraph{Example SB1} We consider the (single element) biquadratic scaled boundary parameterization
\begin{equation}
 \f G^0(u,v) = \left(u (2v - v^2), \, u(1 - v^2)\right)^T
\end{equation}
as in Figure~\ref{fig:examples_scaled-bdr} (left). Table~\ref{tab:example-SB1-1} shos the results when approximating the polynomials $\varphi(x,y)=x^2$ and $\varphi(x,y)=x^3$ and Table~\ref{tab:example-SB1-2} shows the results when approximating the function $\varphi(x,y)=\cos(x)+\sin(y+1)$. All results are summarized in Figure~\ref{fig:example_SB1_scaled-bdr_conv}. The convergence rates derived from Remark~\ref{rem:rates-summary} are also included. Local error contributions are shown in Figure~\ref{fig:example_SB1_scaled-bdr_local_error}.

\begin{figure}[!ht]
    \centering
    \includegraphics[height=0.15\textheight]{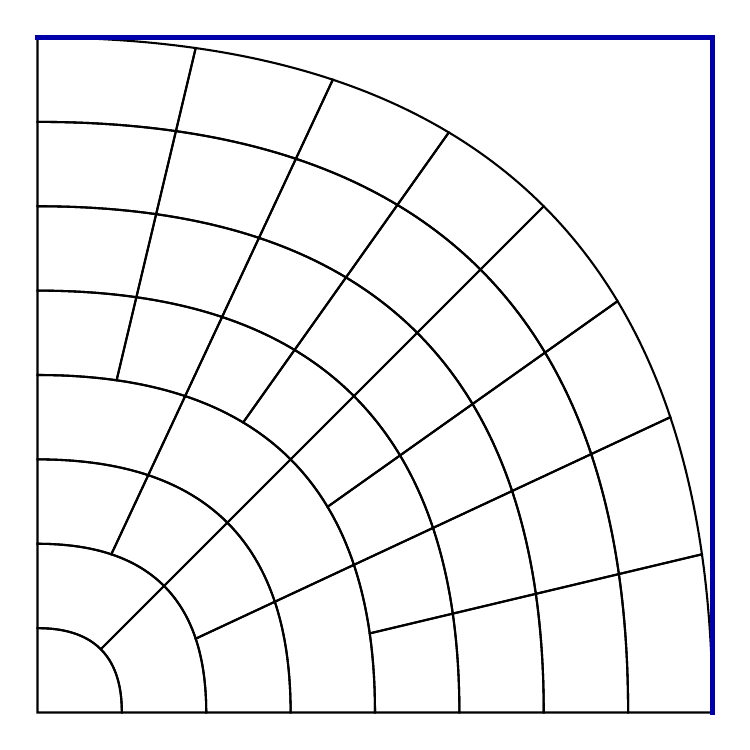} \qquad
    \includegraphics[height=0.15\textheight]{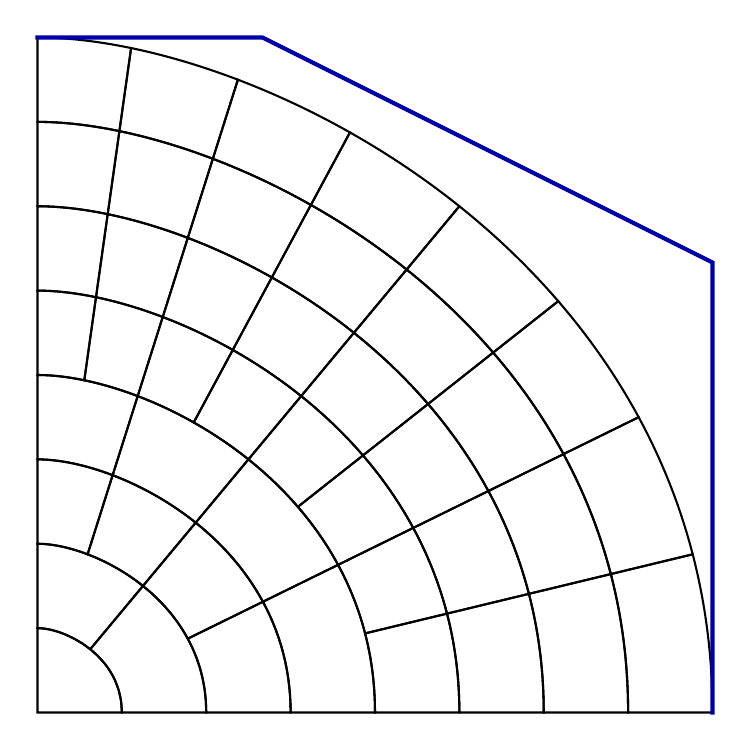}
    \caption{Example SB1 (left): A scaled boundary parameterization with a qaudratic boundary curve. Example A2 (right): A scaled boundary parameterization with a cubic boundary curve. Here we visualize the control polygons of the boundary curves in blue. The depicted mesh is that of level $\ell=2$.}\label{fig:examples_scaled-bdr}
\end{figure}

\begin{table}[ht]
 \centering
 \begin{tabular}{crrrr}
   & $p=2$ & $p=3$ & $p\geq 4$ \\ \hline
  $\ell=0$ & $2^{-8.04491}$ & $2^{-11.4522}$ & $0$ \\
  $\ell=1$ & $2^{-10.7202}$ & $2^{-14.4187}$ & $0$ \\
  $\ell=2$ & $2^{-13.4945}$ & $2^{-17.4105}$ & $0$ \\
  $\ell=3$ & $2^{-16.3223}$ & $2^{-20.4085}$ & $0$
 \end{tabular}\qquad
 \begin{tabular}{crrrr}
   & $p=2$ & $p=3$ & $p=4$ & $p=5$ \\ \hline
  $\ell=0$ & $2^{-8.48577}$ & $2^{-10.7011}$ & $2^{-13.0276}$ & $2^{-16.7594}$ \\
  $\ell=1$ & $2^{-11.1973}$ & $2^{-14.1298}$ & $2^{-16.9851}$ & $2^{-20.7569}$ \\
  $\ell=2$ & $2^{-14.0455}$ & $2^{-17.7971}$ & $2^{-20.9747}$ & $2^{-24.7567}$ \\
  $\ell=3$ & $2^{-16.9888}$ & $2^{-21.5670}$ & $2^{-24.9721}$ & $2^{-28.7567}$
 \end{tabular}
 \caption{$L^2$-errors for Example SB1, approximating the functions $\varphi(x,y)=x^2$ (left) and $\varphi(x,y)=x^3$ (right).}\label{tab:example-SB1-1}
\end{table}
\begin{table}[ht]
 \centering
 \begin{tabular}{crrrr}
   & $p=2$ & $p=3$ & $p=4$ & $p=5$ \\ \hline
  $\ell=0$ & $2^{-9.55256}$ & $2^{-11.5638}$ & $2^{-16.0955}$ & $2^{-18.8649}$ \\
  $\ell=1$ & $2^{-12.0887}$ & $2^{-14.5615}$ & $2^{-20.1640}$ & $2^{-23.3902}$ \\
  $\ell=2$ & $2^{-14.7940}$ & $2^{-17.5507}$ & $2^{-24.2867}$ & $2^{-27.7584}$ \\
  $\ell=3$ & $2^{-17.5759}$ & $2^{-20.5410}$ & $2^{-28.3667}$ & $2^{-31.9793}$
 \end{tabular}
 \caption{$L^2$-errors for Example SB1, approximating the function $\varphi(x,y)=\cos(x)+\sin(y+1)$. As expected, the rates tend to $\lfloor p/2 \rfloor +2$, i.e., $(3,3,4,4)$, instead of the optimal rate $p+1$, i.e., $(3,4,5,6)$.}\label{tab:example-SB1-2}
\end{table}

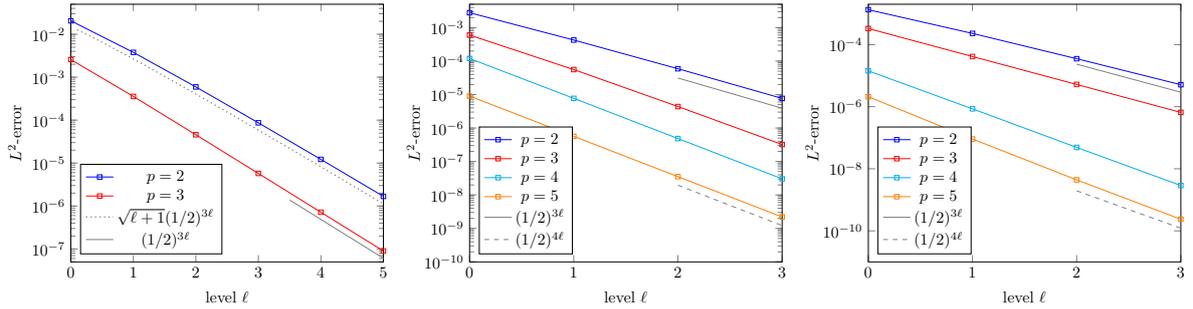
\begin{figure}[!ht]
    \centering
    \begin{tikzpicture}[scale=0.6]
  \begin{axis}[
    xlabel={level $\ell$},
    ylabel={$L^2$-error},
    xmin=0,
    xmax=5,
    ymode=log,
    ymin=0.00000005,
    ymax=0.05,
    grid=none,
    minor tick num=0,
    legend pos=south west
  ]
    \addplot[
      color=blue,
      mark=square,
      mark options={scale=0.7}
    ] coordinates {
      (0, 0.0205502)
      (1, 0.00378653)
      (2, 0.000592784)
      (3, 0.0000866477)
      (4, 0.000012204)
      (5, 0.00000167982)
    };
    \addlegendentry{$p=2$}
    
    \addplot[
      color=red,
      mark=square,
      mark options={scale=0.7}
    ] coordinates {
      (0, 0.00256724)
      (1, 0.000356895)
      (2, 0.0000456615)
      (3, 0.00000573999)
      (4, 0.000000718504)
      (5, 0.0000000898444)
    };
    \addlegendentry{$p=3$}
    
    \addplot[
      color=gray,
      dotted,
      thick,
      domain=0.1:5,
      samples=100
    ] {sqrt(x+1)*0.015*0.5^(3*x)};
    \addlegendentry{$\sqrt{\ell+1}(1/2)^{3\ell}$}
    
    \addplot[
      color=gray,
      domain=3.5:5,
      samples=100
    ] {0.002*0.5^(3*x)};
    \addlegendentry{$(1/2)^{3\ell}$}
  \end{axis}
\end{tikzpicture}
\begin{tikzpicture}[scale=0.6]
  \begin{axis}[
    xlabel={level $\ell$},
    ylabel={$L^2$-error},
    xmin=0,
    xmax=3,
    xtick={0,1,2,3},
    ymode=log,
    ymin=0.0000000001,
    ymax=0.005,
    grid=none,
    minor tick num=0,
    legend pos=south west
  ]
    \addplot[
      color=blue,
      mark=square,
      mark options={scale=0.7}
    ] coordinates {
      (0, 0.00278951)
      (1, 0.000425873)
      (2, 0.0000591384)
      (3, 0.00000768869)
    };
    \addlegendentry{$p=2$}
    
    \addplot[
      color=red,
      mark=square,
      mark options={scale=0.7}
    ] coordinates {
      (0, 0.000600684)
      (1, 0.0000557846)
      (2, 0.00000439065)
      (3, 0.000000321868)
    };
    \addlegendentry{$p=3$}
    
    \addplot[
      color=cyan,
      mark=square,
      mark options={scale=0.7}
    ] coordinates {
      (0, 0.000119754)
      (1, 0.00000770854)
      (2, 0.000000485269)
      (3, 0.0000000303837)
    };
    \addlegendentry{$p=4$}
    
    \addplot[
      color=orange,
      mark=square,
      mark options={scale=0.7}
    ] coordinates {
      (0, 0.00000901386)
      (1, 0.000000564356)
      (2, 0.0000000352761)
      (3, 0.00000000220477)
    };
    \addlegendentry{$p=5$}
    
    \addplot[
      color=gray,
      domain=2:3,
      samples=100
    ] {0.002*0.5^(3*x)};
    \addlegendentry{$(1/2)^{3\ell}$}
    
    \addplot[
      color=gray,
      dashed,
      domain=2:3,
      samples=100
    ] {0.000005*0.5^(4*x)};
    \addlegendentry{$(1/2)^{4\ell}$}
  \end{axis}
\end{tikzpicture}
\begin{tikzpicture}[scale=0.6]
  \begin{axis}[
    xlabel={level $\ell$},
    ylabel={$L^2$-error},
    xmin=0,
    xmax=3,
    xtick={0,1,2,3},
    ymode=log,
    ymin=0.00000000001,
    ymax=0.002,
    grid=none,
    minor tick num=0,
    legend pos=south west
  ]
    \addplot[
      color=blue,
      mark=square,
      mark options={scale=0.7}
    ] coordinates {
      (0, 0.00133166)
      (1, 0.000229583)
      (2, 0.000035201)
      (3, 0.00000511822)
    };
    \addlegendentry{$p=2$}
    
    \addplot[
      color=red,
      mark=square,
      mark options={scale=0.7}
    ] coordinates {
      (0, 0.000330331)
      (1, 0.0000413565)
      (2, 0.00000520841)
      (3, 0.000000655468)
    };
    \addlegendentry{$p=3$}
    
    \addplot[
      color=cyan,
      mark=square,
      mark options={scale=0.7}
    ] coordinates {
      (0, 0.0000142814)
      (1, 0.000000851215)
      (2, 0.0000000488639)
      (3, 0.00000000288922)
    };
    \addlegendentry{$p=4$}
    
    \addplot[
      color=orange,
      mark=square,
      mark options={scale=0.7}
    ] coordinates {
      (0, 0.00000209459)
      (1, 0.0000000909599)
      (2, 0.00000000440445)
      (3, 0.000000000236201)
    };
    \addlegendentry{$p=5$}
    
    \addplot[
      color=gray,
      domain=2:3,
      samples=100
    ] {0.0015*0.5^(3*x)};
    \addlegendentry{$(1/2)^{3\ell}$}
    
    \addplot[
      color=gray,
      dashed,
      domain=2:3,
      samples=100
    ] {0.0000005*0.5^(4*x)};
    \addlegendentry{$(1/2)^{4\ell}$}
  \end{axis}
\end{tikzpicture}
    \caption{Example SB1: $L^2$-error when approximating the functions $\varphi(x,y)=x^2$ (left), $\varphi(x,y)=x^3$ (center) and $\varphi(x,y)=\cos(x)+\sin(y+1)$ (right). Note that the function $\varphi(x,y)=x^2$ can be represented exactly for $p\geq 4$.}\label{fig:example_SB1_scaled-bdr_conv}
\end{figure}

\begin{figure}[!ht]
    \centering
    \includegraphics[height=0.2\textheight]{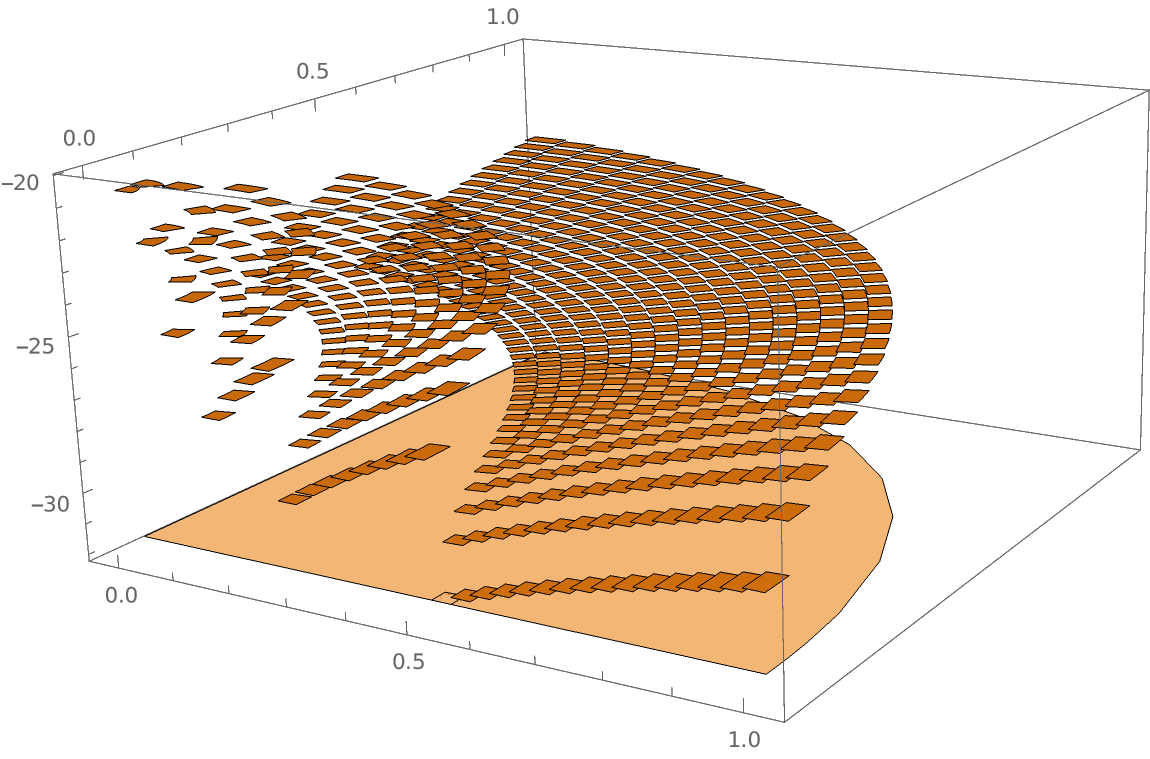} \qquad
    \includegraphics[height=0.2\textheight]{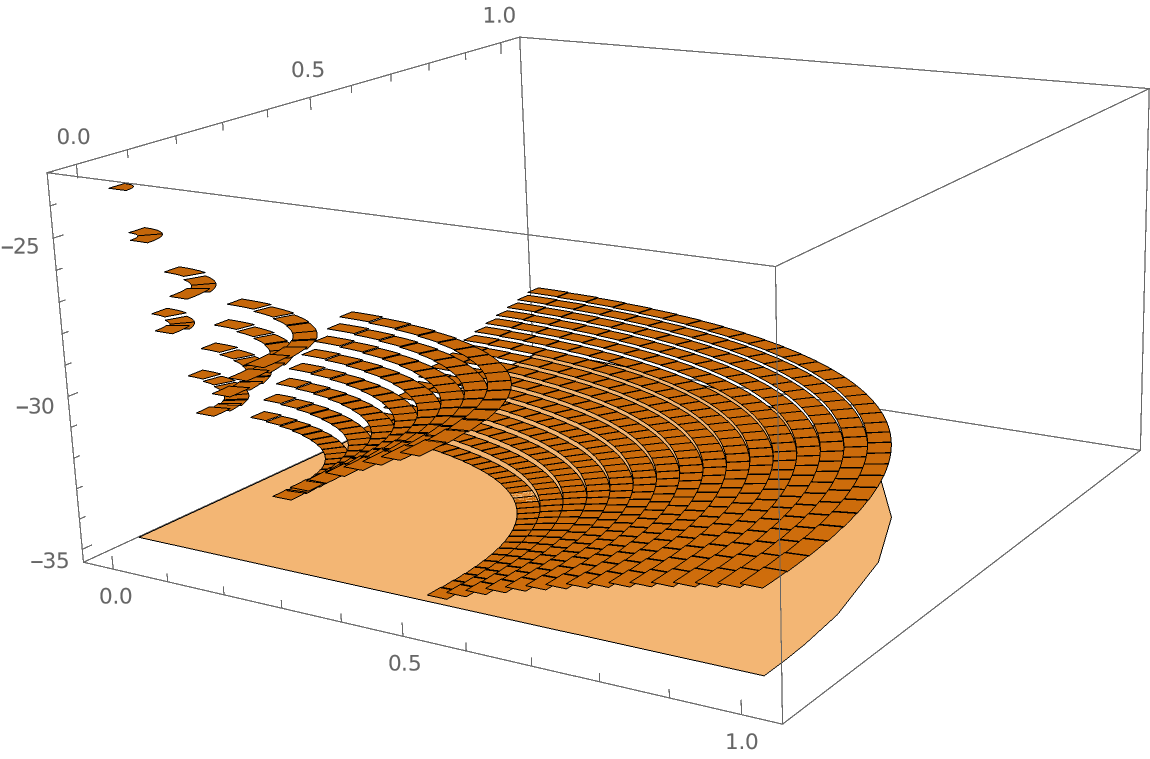}
    \caption{Example SB1: Local contributions of the $L^2$-error when approximating the function $\varphi(x,y)=x^2$ for $p=2$ (left) and $p=3$ (right). The $z$-axis shows the $\log_2$ of the local error on each element, which is equivalent (up to constants) to the $L^\infty$-error of the best $L^2$-approximation.}\label{fig:example_SB1_scaled-bdr_local_error}
\end{figure}

\paragraph{Example SB2} We consider the (single element) bicubic scaled boundary parameterization
\begin{equation}
 \f G^0(u,v) = \left(u (v + v^2 - v^3), \, u(1 - v^2)\right)^T
\end{equation}
as in Figure~\ref{fig:examples_scaled-bdr} (right). The resulting approximation errors for several functions are shown in Figure~\ref{fig:example_SB2_scaled-bdr_conv}.

\begin{figure}[!ht]
    \centering
    \begin{tikzpicture}[scale=0.6]
  \begin{axis}[
    xlabel={level $\ell$},
    ylabel={$L^2$-error},
    xmin=0,
    xmax=3,
    xtick={0,1,2,3},
    ymode=log,
    ymin=0.00001,
    ymax=0.005,
    grid=none,
    minor tick num=0,
    legend pos=south west
  ]
    \addplot[
      color=blue,
      mark=square,
      mark options={scale=0.7}
    ] coordinates {
      (0, 0.00320709233999157)
      (1, 0.0008210398523613419)
      (2, 0.0002064485049519145)
      (3, 0.000051686174497721414)
    };
    \addlegendentry{$p=2$}
    
    \addplot[
      color=gray,
      domain=2:3,
      samples=100
    ] {0.002*0.5^(2*x)};
    \addlegendentry{$(1/2)^{2\ell}$}
  \end{axis}
\end{tikzpicture}
\begin{tikzpicture}[scale=0.6]
  \begin{axis}[
    xlabel={level $\ell$},
    ylabel={$L^2$-error},
    xmin=0,
    xmax=3,
    xtick={0,1,2,3},
    ymode=log,
    ymin=0.00000001,
    ymax=0.005,
    grid=none,
    minor tick num=0,
    legend pos=south west
  ]
    \addplot[
      color=blue,
      mark=square,
      mark options={scale=0.7}
    ] coordinates {
      (0, 0.00425981)
      (1, 0.000662402)
      (2, 0.0000962254)
      (3, 0.0000134972)
    };
    \addlegendentry{$p=2$}
    
    \addplot[
      color=red,
      mark=square,
      mark options={scale=0.7}
    ] coordinates {
      (0, 0.000611935)
      (1, 0.0000786629)
      (2, 0.00000990625)
      (3, 0.00000124062)
    };
    \addlegendentry{$p=3$}
    
    \addplot[
      color=cyan,
      mark=square,
      mark options={scale=0.7}
    ] coordinates {
      (0, 0.0000872531)
      (1, 0.0000109738)
      (2, 0.00000137226)
      (3, 0.000000171537)
    };
    \addlegendentry{$p=4$}
    
    \addplot[
      color=orange,
      mark=square,
      mark options={scale=0.7}
    ] coordinates {
      (0, 0.0000188364)
      (1, 0.00000235484)
      (2, 0.000000294356)
      (3, 0.0000000367945)
    };
    \addlegendentry{$p=5$}
    
    \addplot[
      color=gray,
      domain=2:3,
      samples=100
    ] {0.003*0.5^(3*x)};
    \addlegendentry{$(1/2)^{3\ell}$}
  \end{axis}
\end{tikzpicture}
    \caption{Example SB2: $L^2$-error when approximating the functions $\varphi(x,y)=x$ (left) and $\varphi(x,y)=x^2$ (right). Note that, since $q=3$, the function $\varphi(x,y)=x$ cannot be represented exactly for $p=2$.}\label{fig:example_SB2_scaled-bdr_conv}
\end{figure}
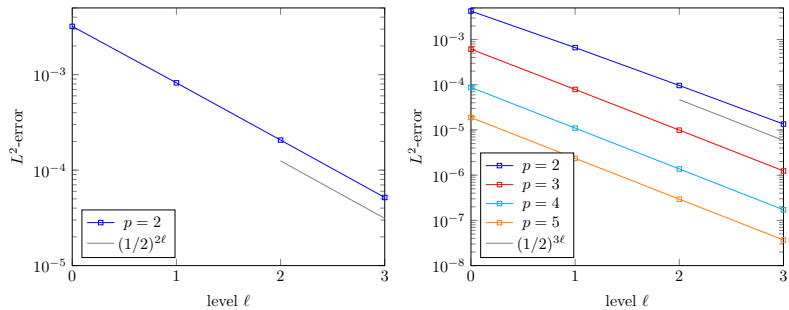

\begin{figure}[!ht]
    \centering
    \includegraphics[height=0.15\textheight]{bad_sector_q3_plot.pdf} \quad
    \includegraphics[height=0.15\textheight]{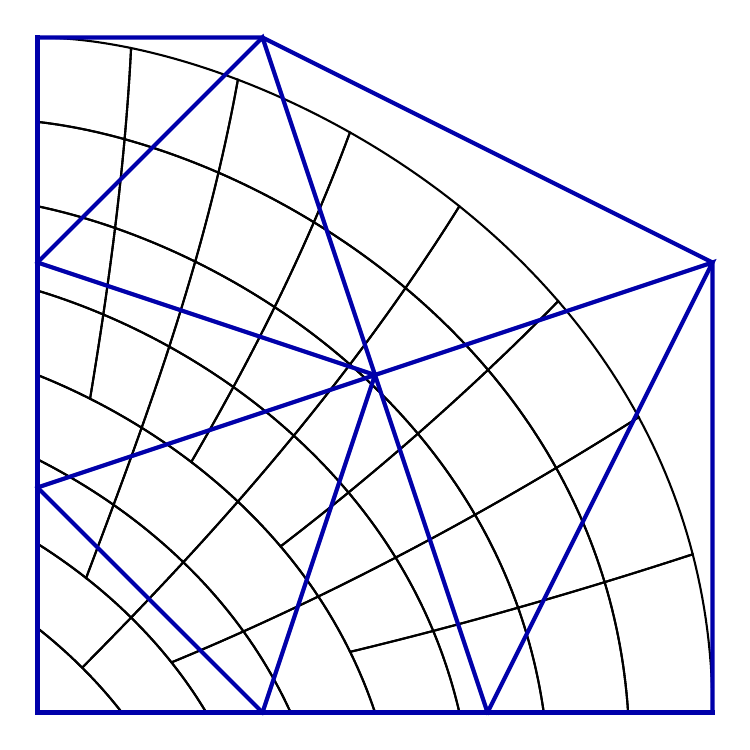} \quad
    \includegraphics[height=0.15\textheight]{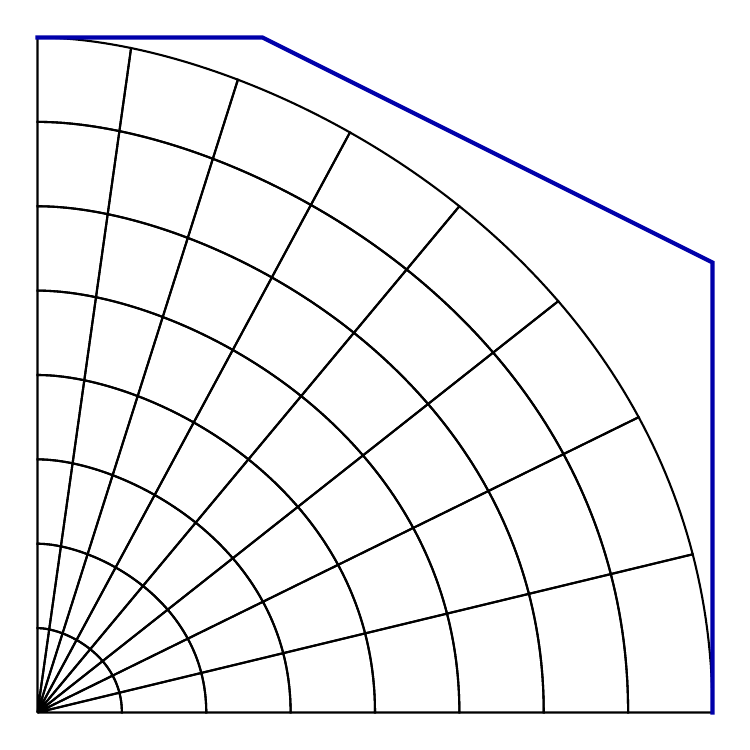} \quad
    \begin{tikzpicture}[scale=0.6]
  \begin{axis}[
    xlabel={level $\ell$},
    ylabel={$L^2$-error},
    xmin=0,
    xmax=3,
    xtick={0,1,2,3},
    ymode=log,
    ymin=0.000000000004,
    ymax=0.005,
    grid=none,
    minor tick num=0,
    legend pos=south west
  ]
    
    \addplot[
      color=blue,
      mark=square,
      mark options={scale=0.7,solid}
    ] coordinates {
      (0, 0.00327286)
      (1, 0.000397127)
      (2, 0.0000494127)
      (3, 0.00000617187)
    };
    \addlegendentry{$p=2$}
    
    \addplot[
      color=red,
      mark=square,
      mark options={scale=0.7,solid}
    ] coordinates {
      (0, 0.000247282)
      (1, 0.0000184981)
      (2, 0.00000121309)
      (3, 0.0000000767434)
    };
    \addlegendentry{$p=3$}
    
    \addplot[
      color=cyan,
      mark=square,
      mark options={scale=0.7,solid}
    ] coordinates {
      (0, 0.0000361774)
      (1, 0.0000012215)
      (2, 0.0000000388599)
      (3, 0.00000000121974)
    };
    \addlegendentry{$p=4$}
    
    \addplot[
      color=orange,
      mark=square,
      mark options={scale=0.7,solid}
    ] coordinates {
      (0, 0.00000236185)
      (1, 0.0000000370061)
      (2, 0.000000000578681)
      (3, 0.00000000000904377)
    };
    \addlegendentry{$p=5$}

    \addplot[
      color=gray,
      domain=2:3,
      samples=100
    ] {0.002*0.5^(3*x)};
    \addlegendentry{$(1/2)^{i\cdot \ell}$}
    
    \addplot[
      color=gray,
      domain=2:3,
      samples=100
    ] {0.00017*0.5^(4*x)};
    
    \addplot[
      color=gray,
      domain=2:3,
      samples=100
    ] {0.00002*0.5^(5*x)};
    
    \addplot[
      color=gray,
      domain=2:3,
      samples=100
    ] {0.000001*0.5^(6*x)};
    
    \addplot[
      color=blue,
      dashed,
      mark=x,
      mark options={scale=0.7,solid}
    ] coordinates {
      (0, 0.00425981)
      (1, 0.000662402)
      (2, 0.0000962254)
      (3, 0.0000134972)
    };
    
    \addplot[
      color=red,
      dashed,
      mark=x,
      mark options={scale=0.7,solid}
    ] coordinates {
      (0, 0.000611935)
      (1, 0.0000786629)
      (2, 0.00000990625)
      (3, 0.00000124062)
    };
    
    \addplot[
      color=cyan,
      dashed,
      mark=x,
      mark options={scale=0.7,solid}
    ] coordinates {
      (0, 0.0000872531)
      (1, 0.0000109738)
      (2, 0.00000137226)
      (3, 0.000000171537)
    };
    
    \addplot[
      color=orange,
      dashed,
      mark=x,
      mark options={scale=0.7,solid}
    ] coordinates {
      (0, 0.0000188364)
      (1, 0.00000235484)
      (2, 0.000000294356)
      (3, 0.0000000367945)
    };
    
    \addplot[
      color=blue,
      dotted,
      thick,
      mark=diamond,
      mark options={scale=0.7,solid}
    ] coordinates {
      (0, 0.00295391)
      (1, 0.00036899)
      (2, 0.0000462094)
      (3, 0.00000577993)
    };
    
    \addplot[
      color=red,
      dotted,
      thick,
      mark=diamond,
      mark options={scale=0.7,solid}
    ] coordinates {
      (0, 0.000219645)
      (1, 0.0000157646)
      (2, 0.00000102131)
      (3, 0.0000000644062)
    };
    
    \addplot[
      color=cyan,
      dotted,
      thick,
      mark=diamond,
      mark options={scale=0.7,solid}
    ] coordinates {
      (0, 0.0000278692)
      (1, 0.000000949318)
      (2, 0.00000003026)
      (3, 0.000000000950238)
    };
    
    \addplot[
      color=orange,
      dotted,
      thick,
      mark=diamond,
      mark options={scale=0.7,solid}
    ] coordinates {
      (0, 0.00000178665)
      (1, 0.0000000285842)
      (2, 0.000000000449097)
      (3, 0.00000000000702649)
    };
    
  \end{axis}
\end{tikzpicture}
    \caption{Example SB2 (left) compared to a B\'ezier-triangle based reparameterization (center left) and to the same geometry with a tensor-product refinement (center right). The control point grid of the underlying B\'ezier-triangle is visualized as well. $L^2$-error when approximating $\varphi(x,y)=x^2$ (right), dashed lines correspond to Example SB2, solid lines to the error on the reparameterization and dotted lines to the error on the tensor-product grid.}\label{fig:examples_scaled-bdr_fix}
\end{figure}

In Figure~\ref{fig:examples_scaled-bdr_fix} we compare the results to two different ways to regain optimal convergence rates. The first is based on a reparameterization of the domain, which is derived from a B\'ezier triangle. This strategy was first proposed in~\cite{takacs2015approximation}. The second approach is to keep the parameterization but to use a standard (singularly mapped) tensor-product grid. While the obtain errors are very similar, the first approach requires a smaller number of elements (here $\frac{2}{3}4^{\ell+1}+\frac{1}{3}$ compared to $4^{\ell+1}$ elements).

\subsection{Characteristic rings for Doo--Sabin subdivision}

In the following we compute $L^2$-approximation errors over characteristic rings for Doo--Sabin subdivision, cf.~\cite{doo1978behaviour}. See Figure~\ref{fig:examples_DS-rings} for visualizations of characteristic rings. The construction of such rings can be found e.g. in~\cite[Section 6.2]{peters2008subdivision}.

\begin{figure}[!ht]
    \centering
    \includegraphics[height=0.2\textheight]{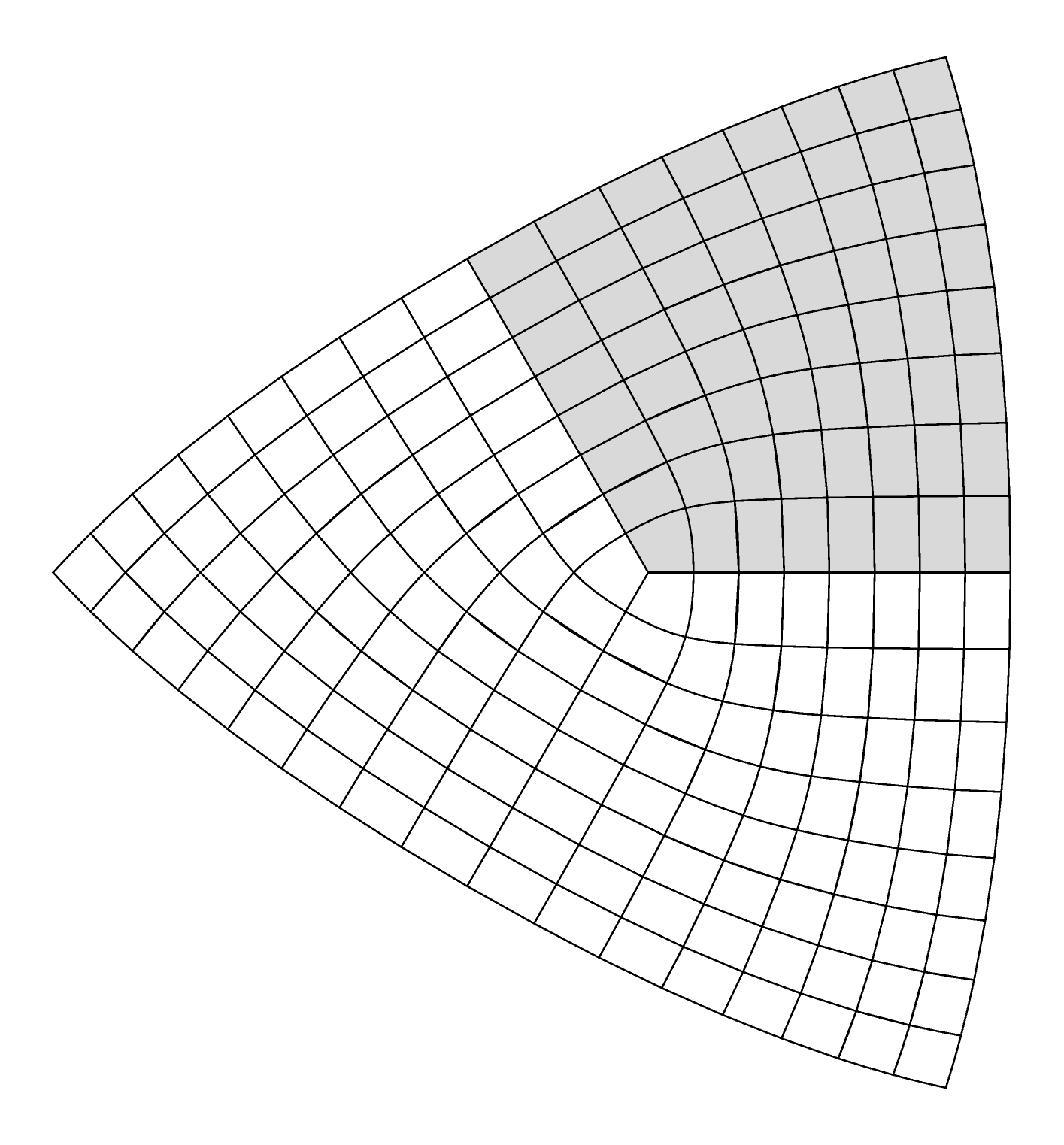} \quad
    \includegraphics[height=0.2\textheight,trim=0 0 0.1cm 0,clip]{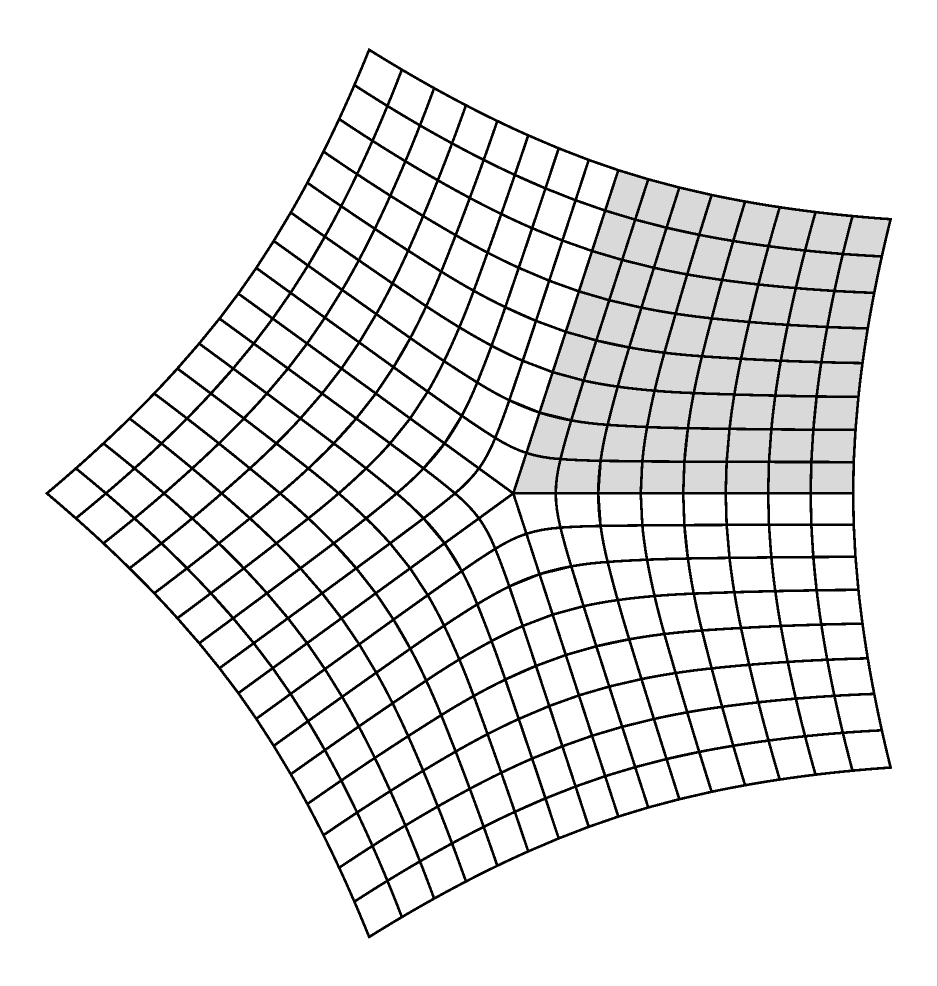} \quad
    \includegraphics[height=0.2\textheight,trim=0 0 0.1cm 0,clip]{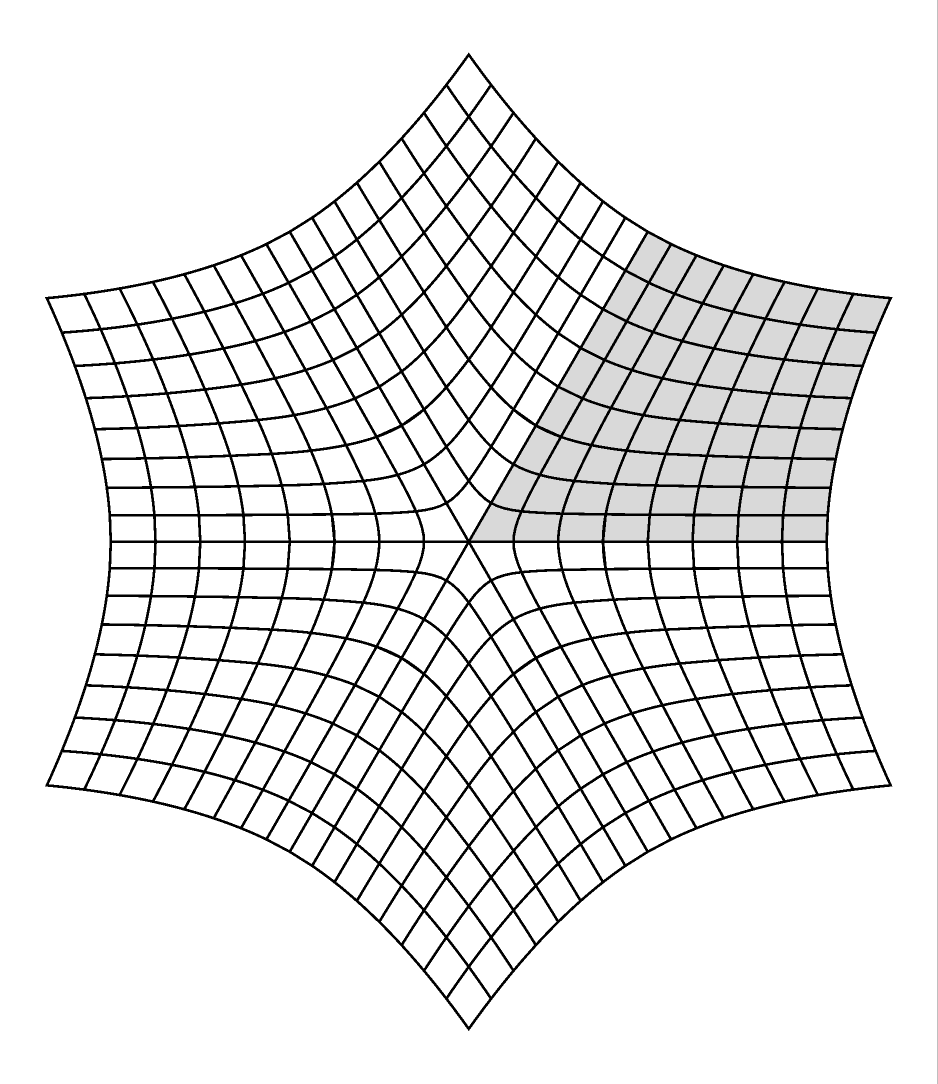} 
    \caption{Characteristic rings of Doo--Sabin subdivision for a vertices of valence three, five and six, respectively. The rings are scaled such that the outermost point on the $x$-axis lies at $(1,0)^T$.}\label{fig:examples_DS-rings}
\end{figure}

We compute approximation errors to given functions $\varphi(x,y)=x^2+y^2$ and $\varphi(x,y)=\sin(x)\cos(y+1)$ and plot the results in Figure~\ref{fig:examples_DS}. The error is always computed only on one sector of the domain as highlighted in gray in Figure~\ref{fig:examples_DS-rings}. Note that in all examples we replace the cap by a Coons-patch, which has the same reproduction degree $\underline{\kappa}_0$ as the neighboring elements.

\begin{figure}[!ht]
    \centering 
    \begin{tikzpicture}[scale=0.6]
  \begin{axis}[
    xlabel={level $\ell$},
    ylabel={$L^2$-error},
    xmin=0,
    xmax=4,
    xtick={0,1,2,3,4},
    ymode=log,
    ymin=0.00000000002,
    ymax=0.005,
    grid=none,
    minor tick num=0,
    legend pos=south west
  ]
    
    \addplot[
      color=blue,
      mark=square,
      mark options={scale=0.7,solid}
    ] coordinates {
      (0, 0.00339509)
      (1, 0.000558479)
      (2, 0.0000833176)
      (3, 0.0000118668)
      (4, 0.00000164499)
    };
    \addlegendentry{$p=2$}
    
    \addplot[
      color=red,
      mark=square,
      mark options={scale=0.7,solid}
    ] coordinates {
      (0, 0.000100043)
      (1, 0.0000134622)
      (2, 0.00000171136)
      (3, 0.000000214804)
      (4, 0.0000000268781)
    };
    \addlegendentry{$p=3$}
    
    \addplot[
      color=cyan,
      dashed,
      mark=square,
      mark options={scale=0.7,solid}
    ] coordinates {
      (0, 0.000012024)
      (1, 0.000000610672)
      (2, 0.0000000302219)
      (3, 0.00000000160568)
      (4, 0.0000000000917669)
    };
    \addlegendentry{$p=4$}
    
    \addplot[
      color=gray,
      domain=3:4,
      samples=100
    ] {0.00002*0.5^(3*x)};
    \addlegendentry{$(1/2)^{3\cdot \ell}$}
    
    \addplot[
      color=gray,
      dotted,thick,
      domain=3:4,
      samples=100
    ] {0.0000035*0.5^(4*x)};
    \addlegendentry{$(1/2)^{4\cdot \ell}$}
    
    \addplot[
      color=blue,
      dashed,
      mark=square,
      mark options={scale=0.7,solid}
    ] coordinates {
      (0, 0.00203495)
      (1, 0.000300144)
      (2, 0.0000419282)
      (3, 0.00000571874)
      (4, 0.000000769153)
    };
    
    \addplot[
      color=red,
      dashed,
      mark=square,
      mark options={scale=0.7,solid}
    ] coordinates {
      (0, 0.000102875)
      (1, 0.00000855863)
      (2, 0.000000729348)
      (3, 0.0000000757345)
      (4, 0.00000000921616)
    };    
  \end{axis}
\end{tikzpicture}
%
%
%
%
%
%
    \begin{tikzpicture}[scale=0.6]
  \begin{axis}[
    xlabel={level $\ell$},
    ylabel={$L^2$-error},
    xmin=0,
    xmax=4,
    xtick={0,1,2,3,4},
    ymode=log,
    ymin=0.000000000001,
    ymax=0.001,
    grid=none,
    minor tick num=0,
    legend pos=south west
  ]
    
    \addplot[
      color=blue,
      mark=square,
      mark options={scale=0.7,solid}
    ] coordinates {
      (0, 0.000494984)
      (1, 0.0000803715)
      (2, 0.0000119224)
      (3, 0.00000169273)
      (4, 0.000000234179)
    };
    \addlegendentry{$p=2$}
    
    \addplot[
      color=red,
      mark=square,
      mark options={scale=0.7,solid}
    ] coordinates {
      (0, 0.0000103108)
      (1, 0.00000138933)
      (2, 0.000000176669)
      (3, 0.0000000221765)
      (4, 0.00000000277495)
    };
    \addlegendentry{$p=3$}
    
    \addplot[
      color=cyan,
      dashed,
      mark=square,
      mark options={scale=0.7,solid}
    ] coordinates {
      (0, 0.000000563734)
      (1, 0.0000000212976)
      (2, 0.000000000801508)
      (3, 0.0000000000503293)
      (4, 0.00000000000371445)
    };
    \addlegendentry{$p=4$}
    
    \addplot[
      color=gray,
      domain=3:4,
      samples=100
    ] {0.000002*0.5^(3*x)};
    \addlegendentry{$(1/2)^{3\cdot \ell}$}
    
    \addplot[
      color=gray,
      dotted, thick,
      domain=3:4,
      samples=100
    ] {0.00000012*0.5^(4*x)};
    \addlegendentry{$(1/2)^{4\cdot \ell}$}
    
        \addplot[
      color=blue,
      dashed,
      mark=square,
      mark options={scale=0.7,solid}
    ] coordinates {
      (0, 0.000222605)
      (1, 0.0000377709)
      (2, 0.00000564969)
      (3, 0.000000796638)
      (4, 0.000000109013)
    };
    
    \addplot[
      color=red,
      dashed,
      mark=square,
      mark options={scale=0.7,solid}
    ] coordinates {
      (0, 0.0000103243)
      (1, 0.000000986527)
      (2, 0.000000105802)
      (3, 0.000000011596)
      (4, 0.00000000130534)
    };
  \end{axis}
\end{tikzpicture}
%
%
%
%
%
%
    \begin{tikzpicture}[scale=0.6]
  \begin{axis}[
    xlabel={level $\ell$},
    ylabel={$L^2$-error},
    xmin=0,
    xmax=4,
    xtick={0,1,2,3,4},
    ymode=log,
    ymin=0.000000000001,
    ymax=0.001,
    grid=none,
    minor tick num=0,
    legend pos=south west
  ]
    
    \addplot[
      color=blue,
      mark=square,
      mark options={scale=0.7,solid}
    ] coordinates {
      (0, 0.000623477)
      (1, 0.000101398)
      (2, 0.0000150592)
      (3, 0.00000213968)
      (4, 0.000000296158)
    };
    \addlegendentry{$p=2$}
    
    \addplot[
      color=red,
      mark=square,
      mark options={scale=0.7,solid}
    ] coordinates {
      (0, 0.0000213792)
      (1, 0.00000288165)
      (2, 0.000000366459)
      (3, 0.0000000460008)
      (4, 0.00000000575612)
    };
    \addlegendentry{$p=3$}
    
    \addplot[
      color=cyan,
      dashed,
      mark=square,
      mark options={scale=0.7,solid}
    ] coordinates {
      (0, 0.000000775943)
      (1, 0.0000000300224)
      (2, 0.0000000013535)
      (3, 0.0000000000944216)
      (4, 0.00000000000684788)
    };
    \addlegendentry{$p=4$}
    
    \addplot[
      color=gray,
      domain=3:4,
      samples=100
    ] {0.000005*0.5^(3*x)};
    \addlegendentry{$(1/2)^{3\cdot \ell}$}
    
    \addplot[
      color=gray,
      dotted, thick,
      domain=3:4,
      samples=100
    ] {0.0000002*0.5^(4*x)};
    \addlegendentry{$(1/2)^{4\cdot \ell}$}
    
    \addplot[
      color=blue,
      dashed,
      mark=square,
      mark options={scale=0.7,solid}
    ] coordinates {
      (0, 0.000265604)
      (1, 0.0000456289)
      (2, 0.00000686166)
      (3, 0.000000971201)
      (4, 0.000000133278)
    };
    
    \addplot[
      color=red,
      dashed,
      mark=square,
      mark options={scale=0.7,solid}
    ] coordinates {
      (0, 0.0000147143)
      (1, 0.00000160642)
      (2, 0.000000181548)
      (3, 0.0000000205223)
      (4, 0.00000000236646)
    };
  \end{axis}
\end{tikzpicture}
%
%
%
%
%
%
    \caption{Convergence rates for $L^2$-approximation on characteristic rings of Doo--Sabin subdivision for valence three (left), five (center) and six (right). Rates for the function $\varphi(x,y)=x^2+y^2$ shown as solid lines and rates for $\varphi(x,y)=\sin(x)\cos(y+1)$ as dashed lines. The error is computed only on the highlighted sector.}\label{fig:examples_DS}
\end{figure}
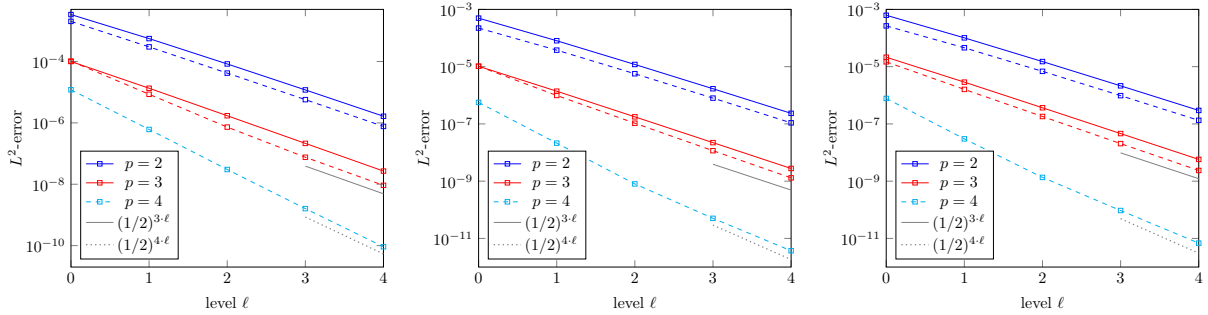

\subsection{Characteristic rings for Catmull--Clark subdivision}

In the following we compute $L^2$-approximation errors over characteristic rings for Catmull--Clark subdivision, cf.~\cite{catmull1978recursively}. See Figure~\ref{fig:examples_CC-rings} for visualizations of characteristic rings. The construction of such rings can be found in~\cite[Section 6.1]{peters2008subdivision}.

\begin{figure}[!ht]
    \centering
    \includegraphics[height=0.2\textheight]{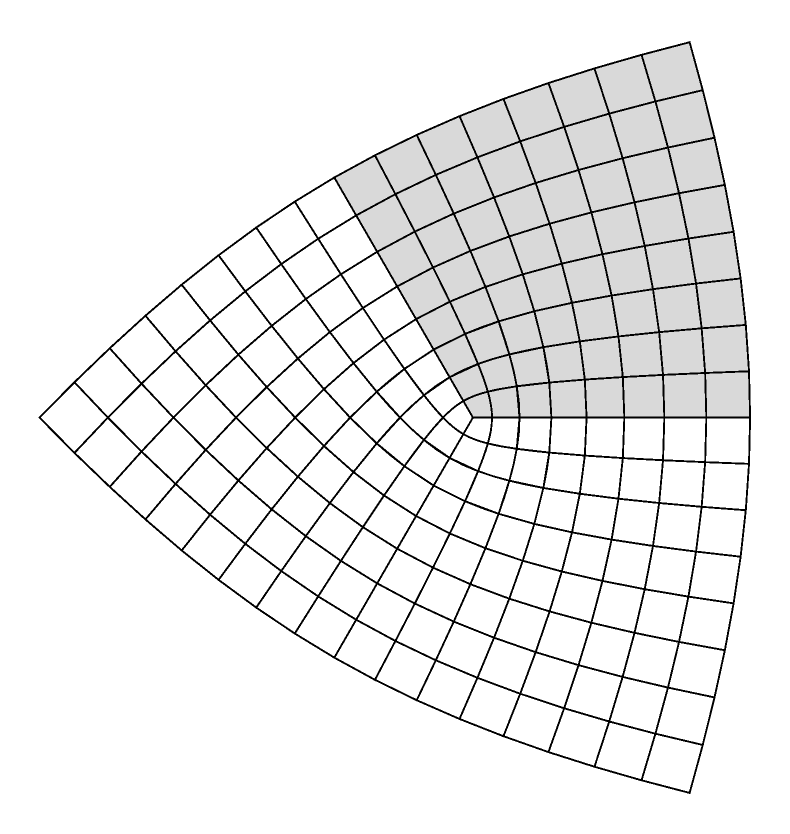} \quad
    \includegraphics[height=0.2\textheight]{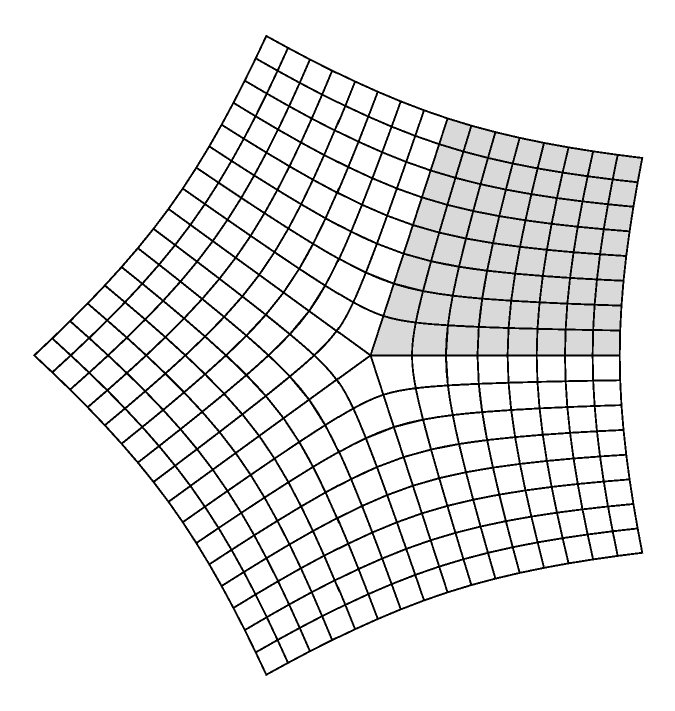} \quad
    \includegraphics[height=0.2\textheight]{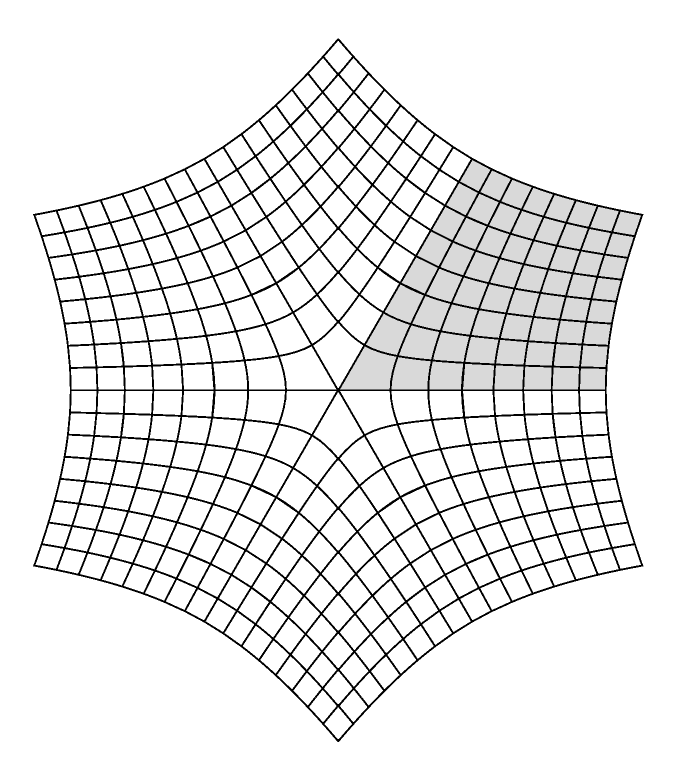} 
    \caption{Characteristic rings of Catmull--Clark subdivision for a vertices of valence three, five and six, respectively. The rings are scaled such that the outermost point on the $x$-axis lies at $(1,0)^T$.}\label{fig:examples_CC-rings}
\end{figure}

Similar to the examples for Doo--Sabin subdivision, we compute approximation errors to the functions $\varphi(x,y)=x^2+y^2$ and $\varphi(x,y)=\sin(x)\cos(y+1)$ and plot the results in Figure~\ref{fig:examples_CC}. Again, the error is computed only on one sector of the domain as highlighted in Figure~\ref{fig:examples_CC-rings} and the cap is replaced by a Coons-patch, which has the same reproduction degree $\underline{\kappa}_0$ as the neighboring elements. The results for valence three and five are moreover summarize in Tables~\ref{tab:examples_CC-3} and~\ref{tab:examples_CC-5}.

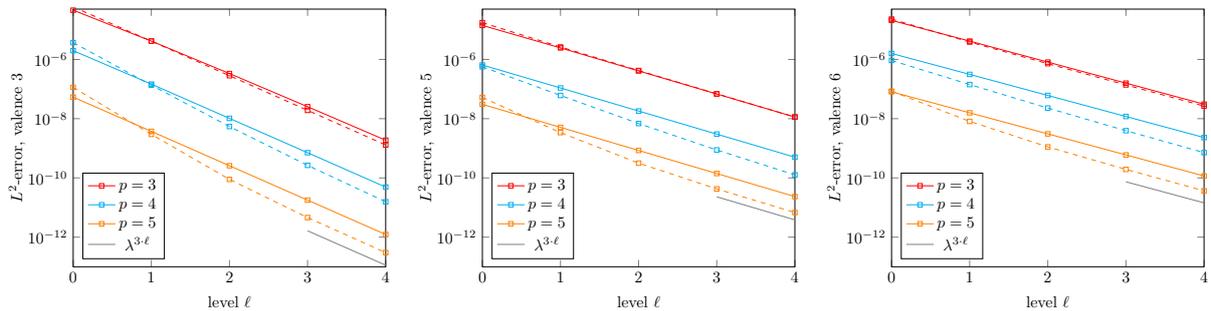
\begin{figure}[!ht]
    \centering 
    \begin{tikzpicture}[scale=0.6]
  \begin{axis}[
    xlabel={level $\ell$},
    ylabel={$L^2$-error, valence $3$},
    xmin=0,
    xmax=4,
    xtick={0,1,2,3,4},
    ymode=log,
    ymin=0.0000000000001,
    ymax=0.00005,
    grid=none,
    minor tick num=0,
    legend pos=south west
  ]
    
    \addplot[
      color=red,
      mark=square,
      mark options={scale=0.7,solid}
    ] coordinates {
      (0, 0.000045967)
      (1, 0.0000041677)
      (2, 0.000000333993)
      (3, 0.0000000253763)
      (4, 0.00000000187249)
    };
    \addlegendentry{$p=3$}
    
    \addplot[
      color=cyan,
      mark=square,
      mark options={scale=0.7,solid}
    ] coordinates {
      (0, 0.00000198163)
      (1, 0.000000146921)
      (2, 0.0000000102725)
      (3, 0.000000000710452)
      (4, 0.0000000000490275)
    };
    \addlegendentry{$p=4$}
    
    \addplot[
      color=orange,
      mark=square,
      mark options={scale=0.7,solid}
    ] coordinates {
      (0, 0.0000000533936)
      (1, 0.0000000037415)
      (2, 0.000000000258258)
      (3, 0.0000000000178128)
      (4, 0.00000000000122867)
    };
    \addlegendentry{$p=5$}
    
    \addplot[
      color=gray,
      domain=3:4,
      samples=100
    ] {0.000000005*0.5^(3.8579*x)};
    \addlegendentry{$\lambda^{3\cdot \ell}$}
        
    \addplot[
      color=red,
      dashed,
      mark=square,
      mark options={scale=0.7,solid}
    ] coordinates {
      (0, 0.0000594171)
      (1, 0.00000415725)
      (2, 0.000000282296)
      (3, 0.0000000191184)
      (4, 0.00000000129775)
    };

    \addplot[
      color=cyan,
      dashed,
      mark=square,
      mark options={scale=0.7,solid}
    ] coordinates {
      (0, 0.00000370418)
      (1, 0.000000135336)
      (2, 0.00000000543739)
      (3, 0.000000000265845)
      (4, 0.0000000000156748)
    };
    
    \addplot[
      color=orange,
      dashed,
      mark=square,
      mark options={scale=0.7,solid}
    ] coordinates {
      (0, 0.000000114028)
      (1, 0.00000000295614)
      (2, 0.0000000000892217)
      (3, 0.00000000000458535)
      (4, 0.000000000000300966)
    };
  \end{axis}
\end{tikzpicture} 
    \begin{tikzpicture}[scale=0.6]
  \begin{axis}[
    xlabel={level $\ell$},
    ylabel={$L^2$-error, valence $5$},
    xmin=0,
    xmax=4,
    xtick={0,1,2,3,4},
    ymode=log,
    ymin=0.0000000000001,
    ymax=0.00005,
    grid=none,
    minor tick num=0,
    legend pos=south west
  ]
    
    \addplot[
      color=red,
      mark=square,
      mark options={scale=0.7,solid}
    ] coordinates {
      (0, 0.0000141075)
      (1, 0.0000024494)
      (2, 0.000000409861)
      (3, 0.0000000682419)
      (4, 0.0000000113543)
    };
    \addlegendentry{$p=3$}
    
    \addplot[
      color=cyan,
      mark=square,
      mark options={scale=0.7,solid}
    ] coordinates {
      (0, 0.000000648408)
      (1, 0.000000108902)
      (2, 0.0000000181236)
      (3, 0.00000000301516)
      (4, 0.000000000501616)    
    };
    \addlegendentry{$p=4$}
    
    \addplot[
      color=orange,
      mark=square,
      mark options={scale=0.7,solid}
    ] coordinates {
      (0, 0.0000000305331)
      (1, 0.00000000508874)
      (2, 0.000000000846599)
      (3, 0.000000000140844)
      (4, 0.0000000000234314)    
    };
    \addlegendentry{$p=5$}
    
    \addplot[
      color=gray,
      domain=3:4,
      samples=100
    ] {0.000000005*0.5^(2.5876*x)};
    \addlegendentry{$\lambda^{3\cdot \ell}$}

    \addplot[
      color=red,
      dashed,
      mark=square,
      mark options={scale=0.7,solid}
    ] coordinates {
      (0, 0.0000174594)
      (1, 0.00000263747)
      (2, 0.000000422304)
      (3, 0.0000000689933)
      (4, 0.0000000113744)
    };

    \addplot[
      color=cyan,
      dashed,
      mark=square,
      mark options={scale=0.7,solid}
    ] coordinates {
      (0, 0.000000564061)
      (1, 0.0000000608434)
      (2, 0.00000000686911)
      (3, 0.000000000879656)
      (4, 0.00000000012667)    
    };
    
    \addplot[
      color=orange,
      dashed,
      mark=square,
      mark options={scale=0.7,solid}
    ] coordinates {
      (0, 0.0000000523428)
      (1, 0.00000000338587)
      (2, 0.000000000315077)
      (3, 0.0000000000425961)
      (4, 0.00000000000678807)    
    };
  \end{axis}
\end{tikzpicture} 
    \begin{tikzpicture}[scale=0.6]
  \begin{axis}[
    xlabel={level $\ell$},
    ylabel={$L^2$-error, valence $6$},
    xmin=0,
    xmax=4,
    xtick={0,1,2,3,4},
    ymode=log,
    ymin=0.0000000000001,
    ymax=0.00005,
    grid=none,
    minor tick num=0,
    legend pos=south west
  ]
    
    \addplot[
      color=red,
      mark=square,
      mark options={scale=0.7,solid}
    ] coordinates {
      (0, 0.000020824)
      (1, 0.00000416435)
      (2, 0.000000813384)
      (3, 0.000000158485)
      (4, 0.0000000308724)
    };
    \addlegendentry{$p=3$}
    
    \addplot[
      color=cyan,
      mark=square,
      mark options={scale=0.7,solid}
    ] coordinates {
      (0, 0.00000159426)
      (1, 0.000000312581)
      (2, 0.0000000608985)
      (3, 0.0000000118626)
      (4, 0.00000000231073)
    };
    \addlegendentry{$p=4$}
    
    \addplot[
      color=orange,
      mark=square,
      mark options={scale=0.7,solid}
    ] coordinates {
      (0, 0.0000000808668)
      (1, 0.0000000157713)
      (2, 0.00000000307213)
      (3, 0.000000000598425)
      (4, 0.000000000116568)
    };
    \addlegendentry{$p=5$}    
    
    \addplot[
      color=gray,
      domain=3:4,
      samples=100
    ] {0.00000001*0.5^(2.35999*x)};
    \addlegendentry{$\lambda^{3\cdot \ell}$}
        
    \addplot[
      color=red,
      dashed,
      mark=square,
      mark options={scale=0.7,solid}
    ] coordinates {
      (0, 0.0000230185)
      (1, 0.00000386903)
      (2, 0.000000719886)
      (3, 0.000000137975)
      (4, 0.0000000267018)
    };

    \addplot[
      color=cyan,
      dashed,
      mark=square,
      mark options={scale=0.7,solid}
    ] coordinates {
      (0, 0.000000926312)
      (1, 0.000000142)
      (2, 0.0000000228074)
      (3, 0.00000000393271)
      (4, 0.000000000714318)
    };
    
    \addplot[
      color=orange,
      dashed,
      mark=square,
      mark options={scale=0.7,solid}
    ] coordinates {
      (0, 0.0000000848484)
      (1, 0.00000000809744)
      (2, 0.00000000111513)
      (3, 0.000000000193985)
      (4, 0.0000000000366864)
    };
  \end{axis}
\end{tikzpicture} 
    \caption{Convergence rates for $L^2$-approximation on characteristic rings of Catmull--Clark subdivision for valence three (left), five (center) and six (right). Rates for the function $\varphi(x,y)=x^2+y^2$ shown as solid lines and rates for $\varphi(x,y)=\sin(x)\cos(y+1)$ as dashed lines. The error is computed only on the highlighted sector. All rates tend to $\lambda^3$, where $\lambda^3 \sim (1/2)^{3.8579}$ for valence three, $\lambda^3 \sim (1/2)^{2.5876}$ for valence five and $\lambda^3 \sim (1/2)^{2.35999}$ for valence six, respectively.}\label{fig:examples_CC}
\end{figure}

\begin{table}[ht]
 \centering
 \begin{tabular}{crrrrrr}
   & $p=3$ ($\log_2$) & $p=3$ ($\log_\lambda$) & $p=4$ ($\log_2$) & $p=4$ ($\log_\lambda$) & $p=5$ ($\log_2$) & $p=5$ ($\log_\lambda$) \\ \hline
  $\ell=0$ & $-14.4090$ & $11.2049$ & $-18.9449$ & $14.7321$ & $-24.1588$ & $18.7865$ \\
  $\ell=1$ & $-17.8723$ & $13.8980$ & $-22.6985$ & $17.6509$ & $-27.9937$ & $21.7687$ \\
  $\ell=2$ & $-21.5137$ & $16.7296$ & $-26.5366$ & $20.6356$ & $-31.8505$ & $24.7678$ \\
  $\ell=3$ & $-25.2319$ & $19.6211$ & $-30.3905$ & $23.6325$ & $-35.7083$ & $27.7678$ \\
  $\ell=4$ & $-28.9924$ & $22.5453$ & $-34.2476$ & $26.6319$ & $-39.5660$ & $30.7676$ 
 \end{tabular}
 \caption{$L^2$-errors for approximating $\varphi(x,y)=x^2+y^2$ on a Catmull--Clark ring of valence three. We can observe that the convergence rates tend to $2^{-3.85789} \sim \lambda^3$, independent of the degree $p \leq 5$. For $p\geq 6$ the function $\varphi(x,y)=x^2+y^2$ can be reproduced exactly and all errors are zero.}\label{tab:examples_CC-3}
\end{table}

\begin{table}[ht]
 \centering
 \begin{tabular}{crrrrrr}
   & $p=3$ ($\log_2$) & $p=3$ ($\log_\lambda$) & $p=4$ ($\log_2$) & $p=4$ ($\log_\lambda$) & $p=5$ ($\log_2$) & $p=5$ ($\log_\lambda$) \\ \hline
  $\ell=0$ & $-16.1132$ & $18.6814$ & $-20.5618$ & $23.8390$ & $-24.9651$ & $28.9441$ \\
  $\ell=1$ & $-18.6391$ & $21.6099$ & $-23.1355$ & $26.8230$ & $-27.5500$ & $31.9411$ \\
  $\ell=2$ & $-21.2184$ & $24.6002$ & $-25.7226$ & $29.8224$ & $-30.1376$ & $34.9411$ \\
  $\ell=3$ & $-23.8048$ & $27.5989$ & $-28.3102$ & $32.8224$ & $-32.7252$ & $37.9411$ \\
  $\ell=4$ & $-26.3922$ & $30.5987$ & $-30.8978$ & $35.8224$ & $-35.3128$ & $40.9411$ 
 \end{tabular}
 \caption{$L^2$-errors for approximating $\varphi(x,y)=x^2+y^2$ on a Catmull--Clark ring of valence five. We can observe that the convergence rates tend to $2^{-2.58758} \sim \lambda^3$, independent of the degree $p \leq 5$. For $p\geq 6$ the function $\varphi(x,y)=x^2+y^2$ can be reproduced exactly and all errors are zero.}\label{tab:examples_CC-5}
\end{table}

In Figure~\ref{fig:example_DSCC} we show a comparison between approximation errors on characteristic rings for Doo--Sabin and Catmull--Clark subdivision. Note that the domains are not exactly the same, but quite similar. For a fixed degree $p$, the approximation on the Doo--Sabin ring is better than on the Catmull--Clark ring. This has two reasons, on the one hand the scaling factor $\lambda_{DS}=0.5 < 0.549988 = \lambda_{CC_5}$, on the other hand the reproduction degree is potentially larger, i.e., $\kappa_{DS} = \lfloor \frac{p}{2} \rfloor \geq \lfloor \frac{p}{3} \rfloor = \kappa_{CC}$. This effect is clearly visible for $p=4$.

\begin{figure}[!ht]
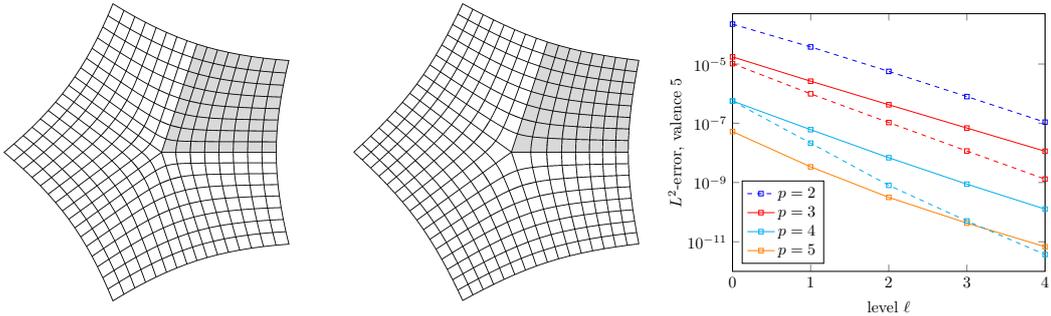

    \centering 
    \includegraphics[height=0.2\textheight,trim=0 0 0.1cm 0,clip]{DS_valence5_mesh.pdf} \quad
    \includegraphics[height=0.2\textheight]{CC_valence5_mesh.pdf}
    \begin{tikzpicture}[scale=0.6]
  \begin{axis}[
    xlabel={level $\ell$},
    ylabel={$L^2$-error, valence $5$},
    xmin=0,
    xmax=4,
    xtick={0,1,2,3,4},
    ymode=log,
    ymin=0.000000000001,
    ymax=0.0005,
    grid=none,
    minor tick num=0,
    legend pos=south west
  ]
    \addplot[
      color=blue,
      dashed,
      mark=square,
      mark options={scale=0.7,solid}
    ] coordinates {
      (0, 0.000222605)
      (1, 0.0000377709)
      (2, 0.00000564969)
      (3, 0.000000796638)
      (4, 0.000000109013)
    };
    \addlegendentry{$p=2$}
        
    \addplot[
      color=red,
      mark=square,
      mark options={scale=0.7,solid}
    ] coordinates {
      (0, 0.0000174594)
      (1, 0.00000263747)
      (2, 0.000000422304)
      (3, 0.0000000689933)
      (4, 0.0000000113744)
    };
    \addlegendentry{$p=3$}

    \addplot[
      color=cyan,
      mark=square,
      mark options={scale=0.7,solid}
    ] coordinates {
      (0, 0.000000564061)
      (1, 0.0000000608434)
      (2, 0.00000000686911)
      (3, 0.000000000879656)
      (4, 0.00000000012667)    
    };
    \addlegendentry{$p=4$}
    
    \addplot[
      color=orange,
      mark=square,
      mark options={scale=0.7,solid}
    ] coordinates {
      (0, 0.0000000523428)
      (1, 0.00000000338587)
      (2, 0.000000000315077)
      (3, 0.0000000000425961)
      (4, 0.00000000000678807)    
    };
    \addlegendentry{$p=5$}

    \addplot[
      color=red,
      dashed,
      mark=square,
      mark options={scale=0.7,solid}
    ] coordinates {
      (0, 0.0000103243)
      (1, 0.000000986527)
      (2, 0.000000105802)
      (3, 0.000000011596)
      (4, 0.00000000130534)
    };
    
    \addplot[
      color=cyan,
      dashed,
      mark=square,
      mark options={scale=0.7,solid}
    ] coordinates {
      (0, 0.000000563734)
      (1, 0.0000000212976)
      (2, 0.000000000801508)
      (3, 0.0000000000503293)
      (4, 0.00000000000371445)
    };
  \end{axis}
\end{tikzpicture} 
    \caption{Comparison between Doo--Sabin (left) and Catmull--Clark (center) subdivision. In the convergence plot (right) errors for Doo--Sabin are shown as dashed lines and for Catmull--Clark as solid lines.}\label{fig:example_DSCC}
\end{figure}

\section{Implications for subdivision based isogeometric analysis}
\label{sec:subdivision-considerations}

In the following we discuss which conclusions can be drawn from the presented results and how this relates to isogeometric discretizations based on subdivision surfaces or subdivision volumes. First of all, note that Theorems~\ref{thm:global-lower} and~\ref{thm:Linfty-estimate} only provide upper bounds for the convergence rates. So, even if the bounds are optimal, the rates may not be attainable.

\paragraph{Reasons for suboptimal convergence}
When approximating a given function $\varphi$ by a function $\varphi_h$ taken from a discretization space $\mathcal{S}^\ell$ which is subspace of $\mathcal{S}[p;\mathcal{M}^\ell_\ast]$, the convergence rate may not be optimal for several reasons:
\begin{enumerate}
 \item The function $\varphi$ may not be sufficiently regular.
 \item The continuity conditions of the space $\mathcal{S}^\ell \subset \mathcal{S}[p;\mathcal{M}^\ell_\ast]$ (e.g. $C^1$- or $C^2$-smoothness on all of $\Omega$ or on the subdomain $\Omega\setminus \omega^{\ell+1}_\ast$) may be too restrictive.
 \item The space $\mathcal{S}[p;\omega^{\ell+1}_\ast]$ at the cap may have suboptimal approximation properties.
 \item The rates in Theorem~\ref{thm:global-lower} or Theorem~\ref{thm:Linfty-estimate} are not optimal.
\end{enumerate}
Away from the extraordinary point, the smoothness conditions do not have a detrimental effect on the approximation. However, at the cap $\omega^{\ell+1}_\ast$ this may be different and the discretization space may be too small. Nonetheless, this can be resolved by enriching the space $\mathcal{S}[p;\omega^{\ell+1}_\ast]$ e.g. by performing locally additional refinement or by increasing the degree locally.
In the following we summarize the results obtained from Theorems~\ref{thm:global-lower} and~\ref{thm:Linfty-estimate}.

\paragraph{Summary of results for Doo--Sabin subdivision}
In case of Doo--Sabin subdivision, we have $\lambda=1/2$ for all valences and $p=2$. For valences $\neq 4$ we have $\underline{\kappa}_0=1$. Thus, following Remark~\ref{rem:rates-summary}, we are in case (b), i.e., $\lambda^{2+\underline{\kappa}_0-r}=1/2^{p+1-r}$. As a consequence, the rates for the $L^\infty$-, $L^2$- and (broken) $H^1$-error for piecewise polynomials are
\[
 1/2^{2\ell} \sim h^2, \qquad \sqrt{1+\ell}/2^{3\ell} \sim \sqrt{1-\log_2(h)} \; h^{3} \qquad \mbox{ and } \qquad \sqrt{1+\ell}/2^{2\ell} \sim \sqrt{1-\log_2(h)} \; h^{2},
\]
respectively. These rates are the best possible for Doo--Sabin subdivision, even if the function spaces are enriched in any $h$-dependent neighborhood of the extraordinary vertex. A summary can be found in Table~\ref{tab:summary-DS}.

\begin{table}[ht]
 \centering
 \begin{tabular}{lrrr}
  & $L^\infty$-rate & $L^2$-rate  & $H^1$-rate \\ \hline
  DS, valence $=4$ & $h^{3}$ & $h^{3}$ & $h^2$ \\
  DS, valence $\neq 4$ & $h^{2}$ & $\sqrt{1-\log_2(h)} \; h^{3}$ & $\sqrt{1-\log_2(h)} \; h^{2}$
 \end{tabular}
 \caption{Best possible convergence rates for the $L^\infty$- and $L^2$-error. The $H^1$-error converges like the $L^\infty$-error. As reference we also give the optimal rates for valence four, which are those of standard biquadratic tensor-product B-splines.}\label{tab:summary-DS}
\end{table}

\paragraph{Summary of results for Catmull--Clark subdivision}
In case of Catmull--Clark subdivision, for valences $\neq 4$, we have $0.410097 < \lambda <\frac{3+\sqrt{5}}{8} \approx 0.6545085$, $\underline{\kappa}_0=1$ and $p=3$. Thus, following Remark~\ref{rem:rates-summary}, we are always in case (a), i.e., $\lambda^{2+\underline{\kappa}_0-r}>1/2^{p+1-r}$. As a consequence, the rates for the $L^\infty$-, $L^2$- and (broken) $H^r$-error for piecewise polynomials are
\[
 \lambda^{2}, \qquad \lambda^{3} \qquad \mbox{ and } \qquad \lambda^{3-r},
\]
respectively. The values of the scaling factor $\lambda$ for common valences are $\lambda \approx 0.410097$ for valence three, $\lambda \approx 0.549988$ for valence five and $\lambda \approx 0.579682$ for valence six. We summarize the expected, best possible convergence rates in Table~\ref{tab:summary-CC}.

\begin{table}[ht]
 \centering
 \begin{tabular}{lrrr}
  & $L^\infty$-rate & $L^2$-rate & $H^1$-rate  \\ \hline
  CC, valence $3$ & $h^{2.57193} \sim 2^{-2.57193}$ & $h^{3.85789} \sim 2^{-3.85789}$ & $h^{2.57193} \sim 2^{-2.57193}$ \\
  CC, valence $4$ & $h^{4} \sim 2^{-4}$ & $h^{4}\sim 2^{-4}$ & $h^{3} \sim 2^{-3}$ \\
  CC, valence $5$ & $h^{2} \sim 2^{-1.72505}$ & $h^{3}\sim 2^{-2.58758}$ & $h^{2} \sim 2^{-1.72505}$ \\
  CC, valence $6$ & $h^{2} \sim 2^{-1.57333}$ & $h^{3}\sim 2^{-2.35999}$ & $h^{2} \sim 2^{-1.57333}$
 \end{tabular}
 \caption{Best possible convergence rates for the $L^\infty$-, $L^2$- and (broken) $H^1$-error. As reference we also give the optimal rates for valence four, which are those of standard bicubic tensor-product B-splines.}\label{tab:summary-CC}
\end{table}

\paragraph{Extension from characteristic rings to general spline rings} Let us consider a planar domain that is constructed through a subdivision procedure from a suitable control mesh. Around each extraordinary vertex a sequence of rings is created. As the mesh is refined, the rings converge to affinely transformed characteristic rings. Thus, the rings are not self-similar, but become more and more self-similar. This is summarized in the following lemma, which derives directly from the asymptotic expansion of rings of subdivision surfaces~\cite[Eq. (5.2)]{peters2008subdivision}. One can show that there exists a limit parameterization $\f G_n^{\lim{}}$ for all elements $\omega_n^i$, such that  
 \[
  \| \f G_n^i - \lambda^i \f G_n^{\lim{}} \|_{L^\infty(B)} = o(\lambda^i),
 \]
as $i$ goes to infinity. Thus, the sequence of element parameterizations $\{\f G_n^i\}_{i=0}^\infty$ is from a compact set of possible element parameterizations, for which, assuming $i$ sufficiently large, the corresponding sequence of constants $\underline{C}_0^i$ has a minimum which is larger zero. Therefore, the approximation properties derived above are also valid for general subdivision rings, not only characteristic rings. There is no reason to believe that for surface domains the behavior is different. The biggest issue there is that at each level one has to consider either the infinite sequence of rings or an approximation which depends on the level.

\paragraph{Modified subdivision schemes}
The results suggest that modifying the subdivision scheme to tune the eigenvalues that are smaller than the subdominant eigenvalue $\lambda$ will not have an effect on the overall convergence behavior, as the convergence rate of the $L^\infty$-error can not be improved by that.

Tuning the subdominant eigenvalue $\lambda$ can be used to improve the convergence rates, but the effect of such a tuning on the actual convergence rates is not clear, since the study in this paper only provides upper bounds on the rates. Shrinking $\lambda$ may have a detrimental effect on errors in higher order Sobolev norms, as the local element curvature can become larger. This should be studied in more detail in future work. The choice $\lambda=0.39$ presented in \cite{ma2019subdivision} corresponds to $\lambda = (0.5^4)^{1/3}$, which results in an optimal $L^2$ rate. The choice $\lambda=0.26 \sim (0.5^4)^{1/2}$ presented in~\cite{wei2021tuned} (almost) optimizes the $L^\infty$ rate. We also refer to the studies~\cite{li2019hybrid,wang2023extended}. The approaches converge optimally for several tested second order problems. However, the effect on higher order problems is not understood yet.

\paragraph{Loop subdivision and other triangle based schemes}
Even though not considered here, triangle based subdivision schemes suffer from the same issues as quadrilateral based ones, as the error estimates for isoparametric finite elements are not specific to the elements being quadrilaterals. Also for triangles, the bounds will depend on the reproducibility of polynomials of a certain total degree over characteristic rings.

\paragraph{Volumetric subdivision schemes}
Even though volumetric subdivision schemes do not fit exactly in the framework discussed in Subsection~\ref{subsec:higher-dim} (because they are not refined by standard bisection near extraordinary edges), they will nonetheless suffer the same reduction in convergence. For isoparametric discretizations over subdivision volumes the $L^\infty$-errors around extraordinary vertices will in general converge with rate $\lambda^2$, independent of the degree of the scheme.

\section{Conclusions}
\label{sec:conclusion}

In this paper we could show that higher order approximation with piecewise polynomial discretization spaces over self-similar meshes of curved finite elements is in general suboptimal. As specific examples we considered scaled boundary parameterizations and characteristic rings of subdivision surfaces. The results extend to any quadrilateral subdivision surfaces where the elements do not converge to parallelograms sufficiently fast. While for scaled boundary parameterizations one can always tune the refinement to any desired $0<\lambda<1$ or use one of the approaches presented in Figure~\ref{fig:examples_scaled-bdr_fix} to obtain optimal convergence rates again, this is not possible for subdivision-like parameterizations.

Even though the results here were presented only for 2D domains, the general ideas and proofs extend directly to higher dimensions. However, the local scaling of elements depends on the self-similarity of rings. Already in 3D, there are several possibilities to construct self similar elements that fill a certain domain. The structure may be cylindrical, a scaled boundary parameterization or a subdivision volume around an extraordinary edge or extraordinary vertex.

When performing isogeometric simulations on subdivision surfaces and volumes, the negative results presented here should be taken into account. Modifying subdivision schemes or enriching the analysis spaces to improve approximation properties are challenging tasks for future research. Alternatively, one may also replace subdivision surface parametrizations (at least locally) be finite, refineable discretizations using geometric continuity, as in~\cite{marsala2022g1}, or multi-patch discretizations, cf.~\cite{hughes2021smooth}.

\section*{Acknowledgements}

I would like to mention that inspirations for this work came from many lengthy discussions with fellow researchers, most importantly with Roland Maier, Philipp Morgenstern, Stefan Takacs and Deepesh Toshniwal. I would like to thank them for their suggestions, without which this paper would not have been possible.

\end{document}